\documentclass[11pt]{amsart}
\usepackage{amsmath,amscd,amssymb,amsfonts,amsthm}

\usepackage[cmtip,matrix,arrow]{xy}

\newtheorem{theorem}{Theorem}[section]
\newtheorem{corollary}[theorem]{Corollary}
\newtheorem{proposition}[theorem]{Proposition}

\newtheorem{lemma}[theorem]{Lemma}
\theoremstyle{definition}
\newtheorem{definition}[theorem]{Definition}
\theoremstyle{remark}

\newtheorem{remark}[theorem]{Remark}

\newtheorem{example}[theorem]{Example}

\renewcommand{\AA}{\mathbb{A}}
\newcommand{\Cour}[1]      {[\![#1]\!]}

\newcommand\G{\mathcal{G}}

\renewcommand{\L}{\mathcal{L}}

\renewcommand{\O}{\mathcal{O}}
\newcommand{\T}{\mathbb{T}}
\newcommand{\J}{\mathcal{J}}

\newcommand{\ca}{\mathcal}

\newcommand{\E}{\ca{E}}

\newcommand{\R}{\mathbb{R}}

\newcommand{\Z}{\mathbb{Z}}

\newcommand\pt{\on{pt}}
\renewcommand{\P}{\mathsf{P}}

\newcommand\lie[1]{\mathfrak{#1}}

\newcommand{\h}{\lie{h}}
\newcommand{\g}{\lie{g}}

\renewcommand{\a}{\mathsf{a}}

\newcommand{\on}{\operatorname}
\newcommand{\Aut}{ \on{Aut} }

\newcommand{\ad}{\on{ad}}

\renewcommand{\subset}{\subseteq}

\renewcommand{\ker}{ \on{ker}}

\newcommand\qu{/\kern-.7ex/} 
\newcommand{\lra}{\longrightarrow}
\newcommand{\hra}{\hookrightarrow}

\renewcommand{\d}{{\mbox{d}}}
\newcommand{\ol}{\overline}

\newcommand{\dd}{\mf{d}}

\newcommand\Om{\Omega}

\newcommand{\f}{\frac}

\newcommand{\ran}{\on{ran}}
\newcommand{\p}{\partial}
\renewcommand{\l}{\langle}
\renewcommand{\r}{\rangle}
\newcommand\hh{{\f{1}{2}}}

\newcommand{\eeq}{\end{eqnarray*}}
\newcommand{\beq}{\begin{eqnarray*}}

\newcommand{\D}{\ca{D}}

\newcommand{\pr}{\on{pr}}
\newcommand{\wh}{\widehat}
\newcommand{\wt}{\widetilde}
\newcommand{\mf}{\mathfrak}

\newcommand{\rra}{\rightrightarrows}

\newcommand{\aut}{\mf{aut}}

\newcommand{\da}{\dasharrow}

\begin{document}
\sloppy
\title[Normal forms around transversals]{Deformation spaces\\ and  normal forms around transversals}
\author{Francis Bischoff}
\address{Department of Mathematics, University of Toronto, 40 St. George Street, Toronto, Ontario,
M5S 2E4}
\email{bischoff@math.toronto.edu}
\email{mein@math.toronto.edu}

\author{Henrique Bursztyn}
\address{IMPA, Estrada Dona Castorina 110, Rio de Janeiro, 22460-320, Brazil.}
\email{henrique@impa.br}

\author{Hudson Lima}
\address{Departamento de Matem\'atica - UFPR,
Centro Polit\'ecnico, Curitiba, 81531-980, Brazil.}
\email{hudsonlima@ufpr.br}

\author{Eckhard Meinrenken}

\begin{abstract}
Given a manifold $M$ with a submanifold $N$, the deformation space $\D(M,N)$ is a manifold with a submersion to $\R$ whose zero fiber is the normal bundle $\nu(M,N)$, and all other fibers are
equal to $M$. This article uses deformation spaces to study the local behavior of various geometric structures associated with singular foliations, with $N$ a submanifold transverse to the foliation. New examples include
$L_\infty$-algebroids, Courant algebroids, and Lie bialgebroids.
In each case, we obtain a normal form theorem around $N$, in terms of a model structure over
$\nu(M,N)$.
\end{abstract}
\maketitle

\section{Introduction}\label{sec:intro}
This article studies the local behavior of various types of geometric structures associated with singular foliations. While much is known about this problem e.g. for Poisson manifolds and Lie algebroids, this paper develops new methods to handle more elaborate situations,  including $L_\infty$-algebroids, Courant algebroids and Lie bialgebroids.

Given a submanifold transverse to the singular foliation, we will prove
normal form theorems identifying the given geometric structure around the submanifold with its linear approximation on the normal bundle. These results imply \emph{local splitting theorems}: Given a leaf
$S\subset M$ of the foliation and a point $m \in S$,  any `small' transversal of complementary dimension passing through $m$ inherits a `transverse structure' for which $m$ is a critical point (in the sense that the transverse structure has $\{m\}$ as a leaf), and the given geometric structure is locally a direct product of the transverse structure and a `trivial' structure along the leaf. More concretely:
\begin{itemize}
\item Any $L_\infty$-algebroid is locally $L_\infty$-isomorphic, near $m\in M$, to a direct product of the tangent bundle $TS$ to the leaf through $m$ and a `transverse' $L_\infty$-algebroid having $m$ as a critical point.

\item Any Courant algebroid is locally isomorphic, near $m\in M$, to a direct product of the standard Courant algebroid $TS\oplus T^*S$ over the leaf through $m$ and a `transverse' Courant algebroid having $m$ as a critical point.

\item Any Lie bialgebroid is locally isomorphic, near $m\in M$,  to a direct product of the standard Lie bialgebroid associated with the symplectic structure on the leaf through $m$ and a `transverse' Lie bialgebroid having $m$ as a critical point.
\end{itemize}
In full generality, our normal form theorems extend these results to neighborhoods of arbitrary transversals.
Our techniques also give new proofs of previously known results of similar nature e.g.
for Poisson manifolds \cite{fre:nor,wei:loc} and Lie algebroids  \cite{bur:spl,duf:nor,fer:con,wei:alm}, and we expect that they apply to many other settings such as Jacobi manifolds, Nijenhuis structures, and so on.

To describe our method in more detail, let $\mf{S}$ denote the geometric structure of interest, determining a singular foliation of the given manifold $M$. Let $N \subseteq M$ be a submanifold 
 transverse to the foliation. We start by defining a  \emph{linear approximation} to $\mf{S}$ along $N$, which is a structure $\nu(\mf{S})$ over the normal bundle $\nu(M,N)$. The desired result is an isomorphism $\tilde{\psi}\colon  \nu(\mf{S}) \to \mf{S}$ covering a tubular neighborhood embedding $\psi \colon \nu(M,N) \to M$.  To prove such a result, our strategy is to construct an interpolating family between the geometric structure $\mf{S}$ and its linear approximation $\nu(\mf{S})$, and then to show that all the members of the family are isomorphic.

The main ingredient to make precise the idea of an interpolating family is the \emph{deformation space} $\D(M,N)$ \cite{hil:mor}.
Recall that this is a manifold with a surjective submersion $\pi\colon \D(M,N)\to \R$, with fibers
\[ \pi^{-1}(t)=\begin{cases} M & \ \ t\neq 0,\\ \nu(M,N) &\ \  t=0.\end{cases}\]
To visualize the deformation space, one may think of the directions normal to $N$ as being `magnified' when $t\to 0$. An important observation is that for the cases at hand, there is a natural
lift $\D(\mf{S}) \to \D(M,N)$ of the geometric structure $\mf{S}$ on $M$ to the deformation space, in such a way that  its pullback
to the fibers $\pi^{-1}(t)$ are given by
\[ \D(\mf{S})|_{\pi^{-1}(t)}=\begin{cases} \mf{S} & \ \ t\neq 0,\\ \nu(\mf{S}) &\ \  t=0.\end{cases}\]
The key step towards the normal form result is to produce an Ehresmann connection on this family of structures, given by an infinitesimal automorphism of
$\D(\mf{S})$ `lifting' the coordinate vector field on $\mathbb{R}$. The desired isomorphism between the fibres of $\D(\mf{S})$ is then given by the parallel transport of this connection. We will construct these connections by means of
\emph{Euler-like sections} of $\mf{S}$, which are lifts of \emph{Euler-like vector fields} on $M$ with respect
to $N$ (that is, complete vector fields vanishing along $N$ and with linear approximation the Euler vector field on $\nu(M,N)$).

Our approach in this paper  was inspired by the work of Haj Saeedi Sadegh  and Higson \cite{hig:eul}, who explained the geometric significance of the fundamental lemma in \cite{bur:spl} asserting that germs of Euler-like vector fields along $N$ are in
1-1 correspondence with germs of tubular neighborhood embeddings $\psi\colon \nu(M,N)\to M$. Viewing the deformation space as an interpolating family between $M$ and the normal bundle of $N$, Haj Saeedi Sadegh and Higson found that every Euler-like vector field along $N$ canonically determines a vector field $W$ on the deformation space. This vector field $W$ lifts the coordinate vector field on $\mathbb{R}$, and its flow induces the corresponding tubular neighbourhood embedding. The main insight of the present paper is that this picture generalizes to a powerful method to prove linearization results for a wide class of geometric structures, including those treated in \cite{bur:spl}, but going much further.

It is interesting to ask for the most general framework in which our methods are applicable. For example, one
might anticipate a general unifying normal form result around transversals in graded geometry or higher stacks with shifted symplectic forms.

\bigskip

The paper is organized as follows. Section \ref{sec:euler} reviews deformation spaces, Euler-like vector fields, and tubular neighborhood embeddings. As an application, we prove a normal form theorem for singular foliations, which generalizes an integrability result of Hermann. In Section \ref{sec:splitla}, we use the deformation space approach to prove a general splitting theorem for anchored vector bundles and Lie algebroids -- this recovers results obtained in \cite{bur:spl}, but the new viewpoint is needed to set the stage for the more complicated structures treated in the subsequent sections. Section \ref{sec:splitinfinity} generalizes the results for Lie algebroids to $L_\infty$-algebroids.  In Section \ref{sec:splitcourant} we prove a normal form theorem for Courant algebroids, as well as for Dirac structures therein. Section~\ref{sec:localsplit} discusses the local structure of Courant algebroids and linear approximations in transverse directions. Finally, Section \ref{sec:splitliebialgebroid} is concerned with a normal form theorem for Lie bialgebroids. Appendix \ref{app:coord} gives coordinate descriptions of deformation spaces, and in  Appendix \ref{app:Moser} we prove a Moser lemma for Lie bialgebroids (needed in Section \ref{sec:splitliebialgebroid}).

\bigskip
\noindent{\bf Acknowledgment.} We are grateful to Nigel Higson, Marco Gualtieri, Alfonso Garmendia and Marco Zambon
for fruitful discussions.
F.B. is supported by an NSERC CGS Doctoral Award. H.B and H.L. thank CNPq, Capes and Faperj for financial support. E.M. is supported by an NSERC Research Grant.
\bigskip

\section{Euler-like vector fields}\label{sec:euler}

This section reviews the relationship between Euler-like vector fields and tubular neighborhood embeddings, which is our main tool to prove splitting theorems. We will mostly follow \cite[Sec.~2]{bur:spl}, but with a different viewpoint based on {\em deformation spaces} from \cite{hig:eul}.
\subsection{Conventions} \label{subsec:convent}
We regard vector fields $X\in\mf{X}(M)$ as derivations $f\mapsto \L_Xf\equiv X(f)$ of the algebra $C^\infty(M)$.
We denote by $\varphi^X_s,\ s\in\R$ the (local) flow,
with the sign convention $\L_X=\f{\p}{\p s}|_{s=0}(\varphi_{-s}^X)^*$ as operators on
functions. For example, the flow of $X=\f{\p}{\p t}$ on $\R$ is given by
$\varphi_s^X(t)=t-s$. Given a vector bundle $E\to M$, we denote by $\mf{aut}_{\ca{VB}}(E)$
the infinitesimal vector bundle automorphisms, consisting of vector fields $\wt{X}\in\mf{X}(E)$ whose (local) flow is by vector bundle automorphisms. By \cite{gra:hig}, these are exactly the vector fields homogeneous of degree $0$ under the scalar multiplication.
As a consequence of the homogeneity, the domain of definition of $\varphi^{\wt{X}}_s$ for a given $s$ is simply the domain of definition of $\varphi_s^X$, where $X$ is the restriction of
$\wt{X}$ to the zero section. In particular, $\wt{X}$ is complete if and only if $X$ is complete.

A vector bundle automorphism $\wt{\varphi}$ of $E$, covering $\varphi: M\to M$, acts on the space of sections $\Gamma(E)$ by $\wt{\varphi}\cdot \sigma = \wt{\varphi}\circ \sigma \circ \varphi^{-1}$. We can use this action to see that
infinitesimal automorphisms $\wt{X}\in \mf{aut}_{\ca{VB}}(E)$ are equivalently expressed as linear operators $D$ on $\Gamma(E)$ for which there exists $X\in\mf{X}(M)$ with
\begin{equation}\label{eq:leibnitz}
D(f\sigma)=f D(\sigma)+(\L_X f)\,\sigma,
\end{equation}
for all $\sigma\in \Gamma(E)$ and $f\in C^\infty(M)$; here $X$, called the {\em symbol} of $D$, is uniquely determined (as the restriction of the corresponding $\wt{X}\in\mf{aut}_{\ca{VB}}(E)$ to the zero section). The relation between $\wt{X}$ and $D$ is
$D (\sigma) =\frac{d}{ds}\big |_{s=0} \varphi^{\wt{X}}_s \cdot \sigma$.

\subsection{Deformation spaces}\label{subsec:def}
In the $C^\infty$-category, the deformation space of a manifold $M$ with respect to a submanifold $N$  (also called the \emph{deformation to the normal cone}) was introduced by Hilsum-Skandalis in \cite[Section 3.1]{hil:mor}, as a generalization of Connes' tangent groupoid \cite{con:non}. The analogous construction in algebraic geometry goes back to work  by MacPherson and Verdier in the 1970s; see the historical comments in \cite[Chapter 4]{ful:int}. Other references for  the material in this section include \cite{hig:eul, hig:tan,kas:ma,wei:hei}.

For any pair $(M,N)$ consisting of a manifold $M$ and a submanifold $i\colon  N\hookrightarrow M$, let $\nu(M,N)=TM|_N/TN$
be the normal bundle:
\begin{equation}\label{eq:pi}
 \xymatrixcolsep{10pt}
\xymatrixrowsep{18pt}
\xymatrix{ \nu(M,N) \ar[d]_p & \\ N\ar[r]_i & M. }
\end{equation}
We denote by $\E\in \mf{X}(\nu(M,N))$ the Euler vector field on $\nu(M,N)$, characterized as the unique vector field such that $\L_\E f=f$ for all linear functions $f\in C^\infty(\nu(M,N))$.

The \emph{deformation space} for $(M,N)$ is a manifold $\D(M,N)$, with a decomposition into two submanifolds:
\begin{equation}\label{eq:decomp}
 \D(M,N)=\nu(M,N)\sqcup (M\times \R^\times).\end{equation}
The manifold structure on $\D(M,N)$ is uniquely determined by the following properties:
\begin{itemize}
\item[(i)] The map \[ \pi\colon \D(M,N)\to \R,\] given by the constant map to $0$ on $\nu(M,N)$ and by $(m,t)\mapsto t$
on $M\times \R^\times$, is a smooth submersion.
\item[(ii)] The map \[ \kappa\colon \D(M,N)\to M,\]
given on $M\times \R^\times$ by $(m,t)\mapsto m$ and on $\nu(M,N)$ by
$i\circ p$, is  smooth.
\item[(iii)] For any function $f\in C^\infty(M)$ with $f|_N=0$, the function $\wt{f}\colon \D(M,N)\to \R$, given on $M\times \R^\times$ by $(m,t)\mapsto \f{1}{t}f(m)$ and on $\nu(M,N)$ by $(v\!\mod TN)\mapsto v(f)$, is  smooth.
\end{itemize}
Intuitively, one may think of the fibers $\pi^{-1}(t)\cong M$ for $t\neq 0$ as being
`stretched in directions normal to $N$' as $t\to 0$, so that the fiber over $0$ is the normal bundle itself. By abuse of notation, we will denote the map $\pi$, regarded as an element of
$C^\infty(\D(M,N))$, simply by `$t$'. The algebra of smooth functions of $\D(M,N)$
is generated by this function $t$, together with functions of the form $\kappa^*f$ for $f\in C^\infty(M)$ and functions $\wt{f}$ for $f|_N=0$. One can use these three types of functions
to define local coordinates on the deformation space; see Appendix~\ref{app:coord}. We will use the notation
\begin{equation} j_t\colon \pi^{-1}(t)\to \D(M,N)\end{equation}
for the inclusions of the fibers of $\pi$; for $t\neq 0$ this may be viewed as an inclusion of $M$.
For $t=0$ we will write
$$
j=j_0\colon \nu(M,N)\to \D(M,N).
$$

We list some properties of the deformation space, which are easily verified using the definition, or using local coordinates (see Appendix~\ref{app:coord}):
\begin{enumerate}
\item\label{it:i} For any morphism $\varphi\colon (M_1,N_1)\to (M_2,N_2)$ (that is, a smooth map
$\varphi\colon M_1\to M_2$ taking $N_1$ into $N_2$) the map $\varphi\times \on{id}\colon M_1\times \R^\times \to M_2\times \R^\times$ extends to a smooth map
\[ \D(\varphi)\colon \D(M_1,N_1)\to \D(M_2,N_2).\]
Its restriction to the zero fiber is the natural map $\nu(\varphi)\colon \nu(M_1,N_1)\to \nu(M_2,N_2)$.
\item\label{it:ii}
For a vector bundle $V\to N$, the map $V\times \R^\times \to V\times \R$,
$(v,t)\mapsto (\f{1}{t}v,t)$, extends to a global diffeomorphism
\[ \D(V,N)\to V\times \R.\]
Its restriction to the zero fiber is the standard identification $ \nu(V,N)\to V$.
\item\label{it:iii}
There is a unique vector field $\Theta\in\mf{X}(\D(M,N))$ such that
\[\L_\Theta t=t,\ \ \L_\Theta \kappa^*f=0,\ \ \L_\Theta \wt{f}=-\wt{f}\ \ \mbox{(for $f|_N=0$).}
\]
This vector field is tangent to the zero fiber, with
$-\E\sim_j \Theta$, while its restriction to $M\times\R^\times$ is the vector field $t\f{\p}{\p t}$.
It generates the $\R^\times$-action on $\D(M,N)$, where $a \in \R^\times$ acts
on $M\times \R^\times$ by $a\cdot
(m,t)=(m,at)$ and on
$\nu(M,N)$ by scalar multiplication by $a^{-1}$.

\item \label{it:v}
Every vector field $Y\in\mf{X}(M)$ determines a vector field
$\widehat{Y}$ on $\D(M,N)$ with the properties
\[ \L_{\widehat{Y}}t=0,\ \ \ \ \L_{\widehat{Y}}\kappa^*(f)=t\kappa^*(\L_Yf), \ \ \ \
\L_{\widehat{Y}}\wt{f}=\kappa^*\L_{Y}f\ \mbox{(for $f|_N=0$). }\]
This vector field is vertical for the submersion $\pi$; its restriction to the zero fiber
 is the equivalence class  $[Y|_N]\in \Gamma(\nu(M,N))$, regarded as a fiberwise constant vector field on $\nu(M,N)$, while the restriction to $\pi^{-1}(t)=M$ for $t\neq 0$ is $tY$.
 \item\label{it:iv}
 If $Y\in\mf{X}(M)$ is tangent to $N$, then the vector field in \eqref{it:v} vanishes along the zero fiber, and hence is divisible by $t$. The vector field $\D(Y)=t^{-1}\wh{Y}\in\mf{X}(\D(M,N))$ is  characterized by
 \[ \L_{\D(Y)}t=0,\ \ \L_{\D(Y)}\kappa^*f=\kappa^*\L_Yf,\ \ \
 \L_{\D(Y)}\wt{f}
 =\wt{\L_Yf}\ \ \mbox{(for $f|_N=0$)}.\]
 Its restriction to the zero fiber is the  \emph{linear approximation} $\nu(Y)$, while the restriction to $M\times \R^\times$ is simply $Y\times 0$.
 (Equivalently, $\nu(Y)$ is obtained by applying the normal bundle functor to $Y\colon (M,N)\to (TM,TN)$, see \cite[Sec.~2.2]{bur:spl}.)
 \item If a differential form $\alpha\in \Omega^k(M)$ vanishes along $N$,  then $\f{1}{t}\alpha\in \Omega^k(M\times \R^\times)$ extends smoothly to a form $\D(\alpha)\in \Omega^k(\D(M,N))$ (see Appendix~\ref{app:coord}). Its restriction to the zero fiber is a linear $k$-form on $\nu(M,N)$,  the \emph{linear approximation $\nu(\alpha)$} of $\alpha$. For $k=0$, i.e., $\alpha=f \in C^\infty(M)$, we have
 $\D(f)=\wt{f}$.
 \end{enumerate}
 %

\subsection{Tubular neighborhood embeddings}\label{subsec:setup}
We define a \emph{tubular neighborhood embedding} for the pair $(M,N)$ to be an embedding
\begin{equation}\label{eq:tub}
 \psi\colon \nu(M,N)\to M,\end{equation}
taking the zero section of the normal bundle to $N\subset M$, and such that the induced map $\nu(\psi)$
is the canonical identification $\nu(V,N)\cong V$ for the vector bundle $V=\nu(M,N)$.
A vector field $X\in\mf{X}(M)$ is called \emph{Euler-like} for $(M,N)$  if it is complete, and
\begin{equation}\label{eq:eulerlikecond}
 X|_N=0,\ \nu(X)=\E.\end{equation}
The conditions \eqref{eq:eulerlikecond} are equivalent to the requirement that if $f\in C^\infty(M)$ vanishes along $N$, then $\L_Xf-f$ vanishes to second order along $N$.  Note that an incomplete vector field with the properties \eqref{eq:eulerlikecond} can be made complete, by multiplying with a function supported on a suitable open neighborhood of $N$, and equal to $1$ near $N$ (see \cite[Rem.~2.9]{bur:spl}). A tubular neighborhood embedding \eqref{eq:tub} determines an Euler-like vector field on its image, by pushing forward the Euler vector field of the normal bundle.
In the other direction, we have \cite[Sec.~2]{bur:spl}:
\begin{theorem} \label{th:euler}
An Euler-like vector field $X$ for $(M,N)$ determines a
unique tubular neighborhood embedding $\psi\colon \nu(M,N)\to M$ with
the property $\E\sim_\psi X$.
\end{theorem}
The proof in \cite{bur:spl} is rather short, but its geometry was further clarified in \cite{hig:eul}, using the deformation space.  We will review the viewpoint from \cite{hig:eul}, since it provides an appropriate framework for the applications to more complicated settings, and since we would like to elaborate on some aspects.

\begin{lemma}\label{lem:W}
\cite{hig:eul} If $X$ is Euler-like, then
the vector field $\f{\p}{\p t}+\f{1}{t} X$ on $M\times \R^\times $ extends to a  vector field  $W\in\mf{X}(\D(M,N))$.
\end{lemma}
\begin{proof}
Since $\E=\nu(X)\sim_j \D(X)$ while $-\E\sim_j \Theta$, the sum $\D(X)+\Theta$ vanishes along the hypersurface $\pi^{-1}(0)$, and is hence divisible by $t$. The vector field $W=\f{1}{t}(\D(X)+\Theta)$ has the desired property.
\end{proof}
See Appendix~\ref{app:coord} for a coordinate description of $W$. We will denote the (local) flow of $W$ by
$\varphi_s=\varphi_s^W$.
\begin{lemma} For all $Y\in\mf{X}(M)$, the vector field $Y+[X,Y]$ is tangent to $N$, and
\begin{equation}\label{eq:yhat}
 [W,\wh{Y}]=\D(Y+[X,Y]).
\end{equation}
If $Y$ is tangent to $N$, then
\begin{equation}\label{eq:dy} [W,\D(Y)]=\f{1}{t} \D([X,Y]).\end{equation}
\end{lemma}
\begin{proof}
If $f\in C^\infty(M)$ vanishes on $N$, then so does $\L_{Y+[X,Y]}f=\L_Y(f-\L_Xf)+\L_X\L_Yf$, since $f-\L_Xf$ vanishes to second order on $N$ and $X|_N=0$. Hence $Y+[X,Y]$ is tangent to $N$. The two commutation relations are evident on $M\times \R^\times$, hence by continuity they hold everywhere on $\D(M,N)$.
\end{proof}

\begin{lemma}
Given $x\in \nu(M,N)\subset \D(M,N)$, the integral curve $\varphi_s(x)$ of $W$ is defined for all $s\in\R$.
\end{lemma}
\begin{proof}
We will show that $\varphi_s(x)$ is defined for all $s>0$ (the proof for $s<0$ is similar).
Since $W\sim_\pi \f{\p}{\p t}$, the map $\pi$ intertwines the flow of $W$ with the translation flow $t\mapsto t-s$ on $\R$. Hence, $\varphi_s(x)\in M\times \R_{<0}$ for $s>0$ sufficiently small. It therefore suffices to show that the solution curves of points $(m,t)\in M\times \R_{<0}\subset \D(M,N)$ are defined for all $s>0$. But on $M\times \R^\times$, the vector field $W$ is given by $\f{1}{t}X+\f{\p}{\p t}$, hence we can describe its (local) flow in terms
of the flow $\varphi_s^X$ of $X$:
\begin{equation}\label{eq:explicit}
 \varphi_s(m,t)
 =(\varphi^X_{-\log(1-\f{s}{t})}(m),\,t-s)
 \end{equation}
for 
$-\infty<\f{s}{t}<1$. In particular, if $t<0$ this is defined for all $s>0$.
\end{proof}

\begin{proof}[Proof of Theorem \ref{th:euler}]
For all $s\in \R$, let $\D(M,N)_s\subset \D(M,N)$ be the domain of $\varphi_s$. By the lemma, this is
an open neighborhood of $\nu(M,N)$ for all $s$. Hence we obtain a well-defined embedding, for any given $s\neq 0$,  by the composition
\begin{equation}\label{eq:sflow}
\nu(M,N)\stackrel{j}{\lra}\D(M,N)_s\stackrel{\varphi_s}{\lra}\D(M,N)\stackrel{\kappa}{\lra}M.
\end{equation}
Equivalently, this map is the restriction of $\varphi_s$ to a map from $\pi^{-1}(0)=\nu(M,N)$ to
$\pi^{-1}(-s) =M$.  Since $[W,\D(X)]=0$, the map \eqref{eq:sflow} intertwines the restrictions of $\D(X)$ to $\nu(M,N)$ and to  $M$; it therefore takes  $\E$ to $X$.

Since $W$ restricts to the vector field $\f{\p}{\p t}$ along $N\times\R\subset \D(M,N)$,
the map \eqref{eq:sflow} restricts to the identity map on $N$.
To find the induced map on normal bundles, consider a section of $\nu(M,N)\cong \nu(\nu(M,N),N)$ represented
as $[Y|_N]$ for some vector field $Y\in \mf{X}(M)$. Recall that the vector field $\wh{Y}\in\mf{X}(\D(M,N))$ is given by $tY$ on $M\times \R^\times$, and by
$[Y|_N]$ (as a fiberwise constant vector field) on $\pi^{-1}(0)=\nu(M,N)$.  Equation
\eqref{eq:yhat} shows that the flow
$\varphi_s$ of $W$ preserves $\wh{Y}$ modulo vector fields that are tangent to $N\times \R$. Hence, the map on normal bundles induced by \eqref{eq:sflow} is
$v\mapsto -s v$, and so is the identity map for $s=-1$. We conclude that the map \eqref{eq:sflow} with $s=-1$,
\begin{equation}
\psi=\kappa\circ \varphi_{-1}\circ j,
\end{equation}
 is the desired tubular neighborhood embedding.
\end{proof}

\begin{remark}
Note that the image $U=\psi(\nu(M,N))$ of the tubular neighborhood embedding $\psi$ is
simply the attracting set of $N$:
$$
U=\{m\in M|\ \ \lim_{s\to \infty}\varphi_s^X(m)\in N\}.
$$
The map $\D(\psi)\colon \nu(M,N)\times \R\equiv \D(\nu(M,N),N)\to
\D(U,N)$ takes $\f{\p}{\p t}$ to $W$; in particular, the restriction of $W$ to
$\D(U,N)\subset \D(M,N)$ is complete.
\end{remark}

\begin{remark}
In \cite[Sec.~2]{bur:spl}, we chose an `initial' tubular neighborhood embedding to reduce to the case that $M$ is a vector bundle $V\to N$, with $N$ its zero section. In terms of the identification $\D(V,N)\cong V$, the vector field $W$ may then be regarded
as a time dependent vector field on $V$. This is exactly the vector field
used in \cite[Lemma~2.4]{bur:spl}.
\end{remark}

We also note the following property of $\psi$. Let $m_t$ denote scalar multiplication by $t$ on $\nu(M,N)$.
\begin{lemma}\label{lem:later}
For all $t\in\R$,
\begin{equation}
\psi\circ m_{t}=
\kappa\circ \varphi_{-t}\circ j.
\end{equation}
\end{lemma}
\begin{proof}
It suffices to prove this for $t=a\neq 0$. Since $W$ is homogeneous of degree $-1$ relative to the $\R^\times$-action on the deformation space (see (c) in Section~\ref{subsec:def}), its flow satisfies
$\varphi_{as}(a\cdot x)=a\cdot \varphi_s(x)$.
Hence
$\kappa(\varphi_{s}(j(m_a(v)))=
\kappa(\varphi_{s}(a^{-1}\cdot j(v)))=\kappa(a^{-1}\cdot \varphi_{as}(j(v)))
=\kappa(\varphi_{as}(j(v)))$.
Now put $s=-1$.
\end{proof}

\subsection{The splitting theorem for singular foliations}\label{subsec:foliation}
We now illustrate these techniques with an application to the theory of singular foliations.
Rather than following the Stefan-Sussmann approach \cite{ste:int,sus:orb}, we will take
the viewpoint of Androulidakis-Skandalis \cite{and:hol} (see also Hermann \cite{her:fol}), and consider a singular foliation on $M$ to be a $C^\infty(M)$-submodule $\J\subset \mf{X}(M)$ such that
\begin{enumerate}
	\item $\J$ is \emph{local}: If a vector field $X\in \mf{X}(M)$ has the property that for all $m\in M$, there
	exists an element of $\J$ equal to $X$ near $m$, then $X\in\J$.
	\item $\J$ is \emph{locally finitely generated}: On some neighborhood of any given $m\in M$, it is spanned by  finitely many vector fields in $\J$.
	\item $\J$ is \emph{involutive}: $[\J,\J]\subset \J$.
\end{enumerate}
Note that  \cite{and:hol} considers $C^\infty(M)$-modules of compactly supported vector fields on $M$; the equivalence to the definition given above is explained in \cite{and:ste} and \cite[Chapter 2]{wan:int}.

A smooth map $\varphi\colon N\to M$ is \emph{transverse to $\J$} if for all $n\in N$, the range of $T_n\varphi$ together with $\{Y_{\varphi(n)},\ Y\in\J\}$ spans the entire tangent space $T_{\varphi(n)}M$. In this case, a singular foliation $\varphi^!\J\subset \mf{X}(N)$ is defined   as the set of all $Y\in \mf{X}(N)$ with the property that  $T\varphi(Y)$ (as a section of the pull-back bundle $\varphi^* TM$) can be written as a locally finite sum $\sum_i h_i X_i$, with $h_i\in C^\infty(N)$ and $X_i\in \J$. See \cite{and:hol}.

If $N$ is a submanifold of $M$ transverse to $\J$, with $i\colon N\to M$ the inclusion,
then $i^!\J$ is simply the set of all $X|_N$ such that $X\in \J$ is tangent to $N$. Letting
$p\colon \nu(M,N)\to N$ be the projection as in \eqref{eq:pi}, we may regard the foliation \[ \nu(\J)=p^!i^!\J\]
of $\nu(M,N)$ as the linear approximation  for the foliation $\J$ around $N$.
\begin{theorem}[Splitting theorem for singular foliations]\label{th:split-foliations}
Let $\J$ be a singular foliation on $M$, and let $N\subset M$ be a submanifold transverse to $\J$.
Then there exists a tubular neighborhood embedding $\psi\colon \nu(M,N)\to M$
such that $\psi^!\J=\nu(\J)$.
Given a proper action of a Lie group $G$ on $M$, preserving the foliation and the transverse submanifold $N$,  one can take the tubular neighborhood embedding to be $G$-invariant.
\end{theorem}
\begin{proof}
We first show that the $C^\infty(M)$-submodule $\J$ contains an Euler-like vector field $X$.
Let  $r=\dim M-\dim N$.
Transversality implies that for all $m\in N$, there is an open neighborhood $V\subset M$
and vector fields $Y_1,\ldots,Y_r\in\J|_V$,
spanning an $r$-dimensional subbundle $K\subset TM|_V$, such that $TM|_{N\cap V}=K|_{N\cap V}\oplus TN|_{N\cap V}$. By \cite[Lemma~3.9]{bur:spl}, there exists a section of $K$ which vanishes along $N$ and is Euler-like. Since such a section is a $C^\infty(V)$-linear combination of the $Y_j$, it lies in
$\J|_V$. This shows how to construct Euler-like vector fields $X_V\in \J|_V$ locally; using a partition of unity we may patch the local definitions to a global Euler-like vector field, $X\in \J$.
Given a $G$-action preserving all the data, one can produce a $G$-invariant $X\in\J$, by averaging
over slices for the action.

The vector fields $\Theta$ and
$\D(X)$ on the  deformation space $\D(M,N)$ lie in $\kappa^!\J$, since
$\Theta\sim_\kappa 0$ and $\D(X)\sim_\kappa X$. Hence, the vector field $W=t^{-1}(\Theta+\D(X))$ also lies in $\kappa^!\J$, which implies that its flow $\varphi_s$
on $\D(U,N)\subset \D(M,N)$ (where $W$ is complete) preserves the foliation \cite[Prop.~1.6]{and:hol} (see also \cite{gar:inn}). Each of the maps
\[
\nu(M,N)\stackrel{j}{\lra} \D(U,N)\stackrel{\varphi_{-1}}{\lra} \D(U,N)\stackrel{\kappa}{\lra} U
\]
is foliation preserving; hence, so is their composition $\psi$.
\end{proof}
\begin{remark}
A \emph{leaf} of a singular foliation $\J$ is a maximal connected submanifold $S\subset M$ with the property that the vector fields in $\J$ are all tangent to $S$, and span all of $TS$. Given $m\in M$, and letting $N$ be a submanifold through $m$ with $T_mN$ a complement to the span of $X_m,\ X\in\J$, we have that $\{m\}$ is a leaf of $\nu(\J)=i^!\J$, and so the fiber $\nu(M,N)_m$ is a
leaf of $p^!i^!\J$. Hence, we recover Hermann's theorem \cite[Theorem 2.2]{her:fol} that every point of $M$ is contained in a (unique) leaf. We also obtain a $G$-equivariant version of this result (for point $m$ fixed under the $G$-action); note that the traditional proofs of Hermann's theorem, by inductively building local normal forms, do not lend themselves to averaging under $G$-actions.
\end{remark}

\section{The splitting theorem for anchored vector bundles}\label{sec:splitla}
In this section, we revisit the splitting theorem for \emph{involutive anchored vector bundles} \cite[Theorem~3.13]{bur:spl}, reformulating its proof in terms of deformation spaces. This will lay the groundwork
for the more complicated cases of $L_\infty$-algebroids and Courant algebroids, to be discussed in the subsequent sections.

An {\em anchored vector bundle} is a vector bundle
 $E\to M$ together with a bundle map $\a\colon E\to TM$ covering the identity map on $M$, called the \emph{anchor}. We denote by $\mf{aut}_{\ca{AV}}(E)$ the infinitesimal vector bundle automorphisms preserving $\a$;  in terms of the associated linear operator $D$ on $\Gamma(E)$ (cf.~ Section \ref{subsec:convent}), the compatibility with the anchor is  expressed by the condition
\[ [X,\a(\tau)]=\a(D\tau),
\]
where $X$ is the symbol of $D$.

The anchored vector bundle $E$  is called \emph{involutive} if the image of its space of sections under the anchor map is a Lie subalgebra $\J=\a(\Gamma(E))\subset \mf{X}(M)$. It follows that it determines a singular foliation on $M$ (see ~Section \ref{subsec:foliation}).
As shown in \cite{bur:spl,lau:uni}, the involutivity condition is equivalent to the existence of a
skew-symmetric bracket $[\cdot,\cdot]$ on $\Gamma(E)$, with the property
that for all $\sigma\in\Gamma(E)$, the operator $[\sigma,\cdot]$ satisfies the Leibniz rule
\[ [\sigma,f\tau]=f[\sigma,\tau]+(\L_{\a(\sigma)}f)\,\tau,\]
as well as
\[ \a([\sigma,\tau])=[\a(\sigma),\a(\tau)].\]
In fact, one may take $[\sigma,\tau]=\nabla_\sigma\tau-\nabla_\tau\sigma$ where $\nabla$ is a torsion-free $\a$-connection (for details, see
\cite[Proposition 3.17]{bur:spl}).
The infinitesimal automorphism defined by $D=[\sigma,\cdot]$, lifting the
vector field $\a(\sigma)$,
 will be denoted by
\begin{equation}\label{eq:tsigma}
\wt{\a}(\sigma)\in\mf{aut}_{\ca{AV}}(E).\end{equation}
\begin{remark}
The structure of an anchored vector bundle, with a skew-symmetric bracket satisfying the Leibniz rule but not necessarily the Jacobi identity, is referred to in the literature as a {\em pre-Lie algebroid} \cite{gra:lie} or \emph{almost Lie algebroid}  \cite{gra:pon,gru:coh}.
In general, the flow of $\wt{\a}(\sigma)\in\mf{aut}_{\ca{AV}}(E)$ need not preserve the bracket $[\cdot,\cdot]$. However, if  $[\cdot,\cdot]$ also satisfies the Jacobi identity,
so that $E$ with these structures is a \emph{Lie algebroid}, then $D=[\sigma,\cdot]$ is a derivation of the bracket, and hence
\begin{equation}\label{eq:tsigma1}
 \wt{\a}(\sigma)\in \mf{aut}_{\ca{LA}}(E),
 \end{equation}
 where $\mf{aut}_{\ca{LA}}(E)$ are the infinitesimal Lie algebroid automorphisms. Hence, in this case the flow of $\wt{\a}(\sigma)$ does preserve the bracket.
\end{remark}

For the remainder of this section, we let $E$ be an involutive anchored vector bundle, and fix a bracket $[\cdot,\cdot]$ on $\Gamma(E)$ as above. If
$i\colon N\hra M$ is a submanifold such that $\a^{-1}(TN)$ is a submanifold
(e.g., if $\a$ is transverse to $N$), then the pull-back $i^!E=\a^{-1}(TN)$ is an anchored subbundle  of $E$. The usual argument for Lie algebroids shows that $i^!E$ inherits a bracket, in such a way that $[\sigma|_N,\tau|_N]=[\sigma,\tau]|_N$ for all sections such that $\a(\sigma),\a(\tau)$ is tangent to $N$. In particular,
$i^!E$ is again involutive.

More generally, given a smooth map $\varphi\colon N\to M$ transverse to the anchor $\a$, one defines the {\em pull-back} $\varphi^!E\to N$ as the
pull-back of $E\times TN$ under inclusion of the graph, $N\cong \on{Gr}(\varphi)\subset M\times N$. Thus, $\varphi^!E$ is the
fiber product
\begin{equation}\label{eq:!map}
\xymatrix{ \varphi^!E\ar[r]^{\varphi_!}\ar[d] & E\ar[d]_{\a}\\ TN\ar[r]_{T\varphi} & TM,}
\end{equation}
where the left vertical map  becomes the anchor of $\varphi^!E$ (we keep denoting it by $\a$), and where the natural map $\varphi_!$ defines a morphism of anchored vector bundles covering the map $\varphi$. The space of sections of $\varphi^!E$ inherits a bracket from the bracket on $\Gamma(E\times TN)$. In particular, $\varphi^!E$ is again involutive. (When $E$ is a Lie algebroid, this is the usual construction of a pull-back Lie algebroid due to Higgins-Mackenzie \cite{hig:alg}, and $\varphi_{!}$ is then a morphism of Lie algebroids.) If $\varphi$ is the projection from a product $N=M\times Q$ to the first factor, we obtain $\varphi^!E=E\times TQ$; this locally describes the pull-back in the case of a submersion. The following notation will be convenient.

\begin{definition}
Given a section $\sigma\in \Gamma(E)$, together with a vector field $Z$ on $N$ satisfying $Z\sim_\varphi \a(\sigma)$, we denote by
\[ \sigma\times_\varphi Z\in \Gamma(\varphi^! E)\]
the section given by restriction of $\sigma\times Z\in \Gamma(E\times TN)$ to $\on{Gr}(\varphi)$.
\end{definition}
Observe that $\a(\sigma\times_\varphi Z)=Z$ and $(\sigma\times_\varphi Z)\sim_{\varphi}\sigma$. 

Suppose now that $i\colon N\to M$ is an inclusion of a submanifold, transverse to the anchor of $E$. Letting $p\colon \nu(M,N)\to N$ be the projection, we will take
\[ \nu(E)=p^!i^!E\to \nu(M,N)\]
to be the \emph{linear approximation}
of the involutive anchored vector bundle $E$ around $N$. It comes with a
distinguished \emph{Euler section}
\begin{equation}\label{eq:eulersection}
 \epsilon=0\times_p \E\in\Gamma(\nu(E)).
\end{equation}
A section $\tau\in \Gamma(E)$,  such that
$\a(\tau)$ is tangent to $N$ has a well-defined linear approximation
\begin{equation}\label{eq:linearapprox}
 \nu(\tau)=\tau|_N\times_p \nu(\a(\tau)),\end{equation} with $\a(\nu(\tau)) = \nu(\a(\tau))$.

%
A section $\sigma\in \Gamma(E)$ will be called \emph{Euler-like} if it vanishes along $N$, and
$\a(\sigma)$  is Euler-like (in particular, $\nu(\sigma)$ is the Euler section $\epsilon$).
The following was proved in \cite[Lemma~3.9]{bur:spl}.
\begin{lemma} \label{EL-ex}
For any submanifold $N\subset M$ transverse to the anchor $\a\colon E\to TM$, there exists an Euler-like section $\sigma\in \Gamma(E)$. Given a proper $G$-action on $M$, preserving $N$ and lifting to an action on $E$, one can take $\sigma$ to be $G$-invariant.
\end{lemma}
The transversality of $\a\colon E\to TM$ to the inclusion $i\colon N\to M$ also implies its transversality  to the map $\kappa\colon \D(M,N)\to M$. Hence we obtain an involutive anchored vector bundle
\[ \D(E)=\kappa^!E\to \D(M,N)\]
over the deformation space, with
\begin{equation}\label{eq:open}
\D(E)|_{M\times \R^\times}=E\times T\R^\times,\ \ \ j^!\D(E)=\nu(E).\end{equation}
This follows because over $M\times \R^\times$, the map $\kappa$ is just the projection, and
since its restriction to the zero fiber is $i\circ p$. We will need two types of sections of $\D(E)$:
\begin{enumerate}
\item\label{it:2c} The vector field $\Theta$ on the deformation space gives a section
\[
\theta=0\times_\kappa \Theta\in \Gamma(\D(E)),
\]
with $\a(\theta)=\Theta$. In terms of \eqref{eq:open}, its restriction to $M\times\R^\times$
is the section $0\times t\f{\p}{\p t}$ of $E\times T\R^\times$, while its restriction to
$\nu(M,N)$ is minus the Euler section, $-\epsilon$.
\item\label{it:2d} A section $\tau\in \Gamma(E)$,  such that
$\a(\tau)$ is tangent to $N$, defines a section
\[ \D(\tau)=\tau\times_\kappa \D(\a(\tau)) \in \Gamma(\D(E)),\]
with $\a(\D(\tau))=\D(\a(\tau))$. Its restriction to $M\times \R^\times$ is $\tau\times 0$, while the restriction to $\nu(M,N)$ is $\nu(\tau)$.
\end{enumerate}

\begin{lemma} If $\sigma\in \Gamma(E)$ is Euler-like, then the section $\f{1}{t}\sigma+\f{\p}{\p t}$ of
$E\times T\R^\times\cong \D(E)|_{M\times \R^\times}$ extends smoothly to a section
\[
w\in \Gamma(\D(E))
\]
satisfying $\a(w)=W$ and $[w,\D(\sigma)]=0$.
\end{lemma}
\begin{proof}
By the above, $\D(\sigma)|_{\nu(M,N)}=\nu(\sigma)=\epsilon$. Hence, the section $\D(\sigma)+\theta$ satisfies \[ (\D(\sigma)+\theta)|_{\nu(M,N)}=
\epsilon-\epsilon=0;\]
it is therefore divisible by $t$.
\end{proof}

The section $w$ defined in the previous lemma gives rise to an infinitesimal automorphism (see \eqref{eq:tsigma})
\[
\wt{W}=\wt{\a}(w)\in\mf{aut}_{\ca{AV}}(\D(E))
\]
with base vector field $W=\a(w)$. Since $W$ is complete over $\D(U,N)$, the vector field
$\wt{W}$ is complete on $\D(E|_U)=\D(E)|_{\D(U,N)}$. Its flow $\wt{\varphi}_s$ is by automorphisms of anchored vector bundles, with base map $\varphi_s$ the flow of $W$. The map $\varphi_{-1}$
intertwines the inclusion $j_0$ of $\nu(M,N)=\pi^{-1}(0)$ with the inclusion $j_1$ of
$U=\pi^{-1}(1)\cap \D(U,N)$. Hence it induces a unique isomorphism of anchored vector bundles
$\wt{\psi}\colon j_0^!\D(E|_U)\to j_1^!\D(E|_U)$, such that
$$(j_1)_!\circ \wt{\psi}=\wt{\varphi}_{-1}\circ (j_0)_!.$$
Using the canonical identifications
$j_0^!\D(E|_U)\cong \nu(E)$, and
$j_1^!\D(E|_U)\cong E|_U$, this is the desired isomorphism  of anchored vector bundles:
\begin{equation}\label{eq:LAsplit11} \xymatrix{ \nu(E)\ar[r]^{\wt{\psi}} \ar[d] & E|_U\ar[d]\\ \nu(M,N)\ar[r]_{\psi} & U.
}
\end{equation}
Since $\kappa\circ j_1=\on{id}_M$, the morphism $\wt{\psi}$ may alternatively be described
as the following composition of morphisms of anchored vector bundles:
\begin{equation}\label{eq:LAsplit12} \xymatrix@C=8ex{ \nu(E)\ar[r]^{{j}_!}\ar[d] & \D(E|_U) \ar[r]^{\wt{\varphi}_{-1}}\ar[d] & \D(E|_U) \ar[r]^{{\kappa}_!}\ar[d]
& E|_U \ar[d]\\
\nu(M,N) \ar[r]_j  & \D(U,N) \ar[r]_{\varphi_{-1}} & \D(U,N) \ar[r]_\kappa & U.
}
\end{equation}
Here the maps ${\kappa}_!,\ j_!$ comes from the definition $\D(E)=\kappa^!E$ (see \eqref{eq:!map}) and
the identification $\nu(E)= j^!\D(E)$.

An additional property of this construction, not explicit in \cite{bur:spl}, is the following.
Given a section $\tau$ with $\a(\tau)$ tangent to $N$ and $[\sigma,\tau]=0$, we have
$[w,\D(\tau)]=0$. (This is obvious over $M\times\R^\times$, and hence holds globally.)
It follows that the flow  $\wt{\varphi}_s$ preserves the section $\D(\tau)$, and therefore $\wt{\psi}\circ \nu(\tau)= \tau|_U\circ \psi$, or
\[ \wt\psi^*\tau=\nu(\tau),\]
where we put $\wt\psi^*\tau=\wt{\psi}^{-1}\circ \tau\circ \psi$.
In particular,
\[ \wt\psi^*\sigma=\epsilon.\]
We also refer to the isomorphism \eqref{eq:LAsplit11} as a \emph{splitting theorem}, since a (local) trivialization of the normal bundle $\nu(M,N)=N\times \mathsf{P}$ identifies $\nu(E)=p^!i^!E$ with a product $i^!E\times T\mathsf{P}$. In particular, given $m\in M$ one may take $N$ to be a `small' submanifold passing through $m$ and with $T_mN$ complementary to the range of the anchor map at $m$.

Throughout this discussion, we did not assume the Jacobi identity for
the bracket $[\cdot,\cdot]$. If the Jacobi identity holds, so that $E$ is a Lie algebroid, then $\wt{\a}(w)$ is an infinitesimal Lie algebroid automorphism. Hence, the flow $\wt{\varphi}_s$ is by Lie algebroid automorphisms, and the resulting map $\wt{\psi}\colon \nu(E)\to E|_U$ is an isomorphism of Lie algebroids (see \cite{fre:sub} for a related construction). For $N$ a small transversal at a given point $m\in M$,  one recovers the splitting theorem for Lie algebroids due to Fernandes, Weinstein, and Dufour  \cite{duf:nor,fer:con,wei:alm}. Among the advantages of our coordinate-free approach are its
functorial properties; as a consequence, one immediately obtains the $G$-equivariant version (for proper actions lifting to Lie algebroid automorphisms), as well as versions for $\ca{VB}$-algebroids, $\ca{LA}$-groupoids, or Lie algebroid representations.

\begin{remark}
The method just described to obtain splitting theorems for involutive anchored vector bundles and Lie algebroids, based on deformation spaces, can also be used in other contexts. For instance, suppose that $\G\rra M$ is a Lie groupoid, and $i\colon N\hra M$ a submanifold transverse to the orbits of $\G$ (cf. \cite[Sec.~4.4]{bur:spl}). Let $E$ be the associated Lie algebroid, and $\sigma\in \Gamma(E)$ a section that is Euler-like. The groupoid $\D(\G)=\kappa^!\G$ has Lie algebroid $\D(E)=\kappa^! E$, and
the section $\D(\sigma)$ defines a vector field on $\D(\G)$, whose local flow is by
groupoid automorphisms. Letting $\nu(\G)=p^!i^!\G$ be the linear approximation of $\G$, one obtains a lift of the tubular neighborhood embedding $\psi$ defined by $\a(\sigma)$ to a Lie groupoid isomorphism
\[ \nu(\G)\to \G|_U,\]
as ${j}_!\colon \nu(\G)\to \kappa^!\G|_{\D(U,N)}$, followed by the time $-1$-flow, followed by ${\kappa}_!$ (the maps ${j}_!$ and ${\kappa}_!$ are defined just as the analogous maps in \eqref{eq:LAsplit12}).
 \end{remark}

\section{The splitting theorem for $L_\infty$-algebroids}\label{sec:splitinfinity}
In recent years, $L_\infty$-algebroids have become increasingly significant
in mathematics and physics. One motivation comes from the theory of singular foliations: as shown in \cite{lau:uni}, Lie-$n$ algebroids provide
a notion of \emph{resolution} of a singular foliation, with $n=0$ corresponding to regular foliations, and $n=1$  corresponding to
foliations coming from Lie algebroids. They also play a role in
string theory \cite{sta:twi}, the deformation
theory of coisotropic submanifolds \cite{oh:def}, higher Poisson geometry \cite{bru:alg},
shifted symplectic geometry \cite{pym:sh}, and representations up to homotopy \cite{jot:lie2}.

\subsection{Definitions}
The shortest definition of an $L_\infty$-algebroid is in super-geometric terms, as
\emph{split $NQ$-manifolds}, i.e., split non-negatively graded manifolds with a homological vector field of degree 1 \cite{sev:so,she:hig,vor:q}. In more classical terms, this is the definition in terms of
the associated \emph{Chevalley-Eilenberg complex}. Consider a non-positively graded vector bundle
\[ E=\bigoplus_{r\le 0}E_r\to M,\]
where each $E_r$ has finite rank (but $E$ itself might be of infinite rank). Let $E[1]$ be the graded bundle with the shifted grading $E[1]_r=E_{r+1}$, and let $E[1]^*=E^*[-1]$ be its (graded) dual, with grading
\[ E[1]^*=\bigoplus_{r> 0}(E[1]^*)_r\to M,\]
where
\[ (E[1]^*)_r=(E[1]_{-r})^*=(E_{1-r})^*.\]
We denote by $\mathsf{C}(E)\to M$ the symmetric coalgebra bundle over $E[1]$ (cf.~ \cite[Appendix A]{cat:rel}); its (graded) dual is the symmetric algebra bundle
\[ \mathsf{C}(E)^*=\on{Sym}(E[1]^*)\to M.\]
Thus, the fibers of $\mathsf{C}(E)^*$ are generated by homogeneous elements in fibers of $E[1]^*$, subject to relations
$xx'=(-1)^{rr'} x'x$ whenever $x\in (E[1]^*)_r,\ x'\in (E[1]^*)_{r'}$.
The grading of $\mathsf{C}(E)^*$ is such that the inclusion of $E[1]^*$ is degree preserving. The algebra bundle $\mathsf{C}(E)^*$ also carries another grading coming from the symmetric power degree; we will adopt the terminology of \cite{pym:sh} and call this the \emph{weight}. The subspace of elements of weight
$n$ is a graded subspace $C^n(E)^*=\oplus_{r\ge 0} C^n(E)^*_r$; note that
\begin{equation}\label{eq:zero}
C^n(E)^*_r=0 \mbox{\ \ \ for \ \ } r<n.\end{equation}
Clearly, $C^0(E)^*=M\times \R$ and $C^1(E)^*=E[1]^*$. Every section $\sigma\in \Gamma(E[1]_r)=\Gamma(E_{r+1})$ defines a degree $r$ derivation $\iota(\sigma)$ of the algebra $\Gamma(\mathsf{C}(E)^*)$ lowering the weight by $1$; on $\Gamma(E[1]^*)$ it is the obvious pairing.

\begin{definition}\begin{enumerate}
\item
An \emph{$L_\infty$-algebroid structure} on $E\to M$ is a degree $1$
derivation $\d$ of the algebra $\Gamma(\mathsf{C}(E)^*)$ with $\d\circ \d=0$.
If $E_{1-r}=0$ for $r>k$, then $E$ is also called a $\on{Lie}_k$- (or $L_k$-) algebroid.
\item
Suppose $E\to M$ and $F\to N$ are
$L_\infty$-algebroids. A \emph{morphism of $L_\infty$-algebroids}
\[ \Phi\colon F\da  E\]
is given by a morphism of graded vector bundles
\[\Phi\colon \mathsf{C}(F)\to \mathsf{C}(E)\]  such that the pullback map on sections of the dual bundle $\Phi^*\colon \Gamma(\mathsf{C}(E)^*)\to \Gamma(\mathsf{C}(F)^*)$ is a morphism of differential graded algebras (cf. \cite[Thm.~3]{bon:ca}).
\end{enumerate}
\end{definition}
\begin{example}
If $M=\pt$, the definition of $L_\infty$-algebroid reduces to that of an
\emph{$L_\infty$-algebra}. On the other hand, if $E_r=0$ for $r\neq 0$ one recovers the definition of a  Lie algebroid in terms of its Chevalley-Eilenberg complex \cite{vai:lie}.
The anchor map $\a\colon E\to TM$ of an $L_\infty$-algebroid
(see \eqref{eq:linftyanchor} below)
is an example of an $L_\infty$-morphism.
\end{example}

Unpacking the data of an $L_\infty$-structure on $E=\bigoplus_{r\le 0} E_r$
reveals a rich geometric structure \cite{bon:ca,she:hig}. Let  $\d_n$ be the component of $\d$ raising the weight by $n$:
\[ \d_n\colon \Gamma(E[1]^*)\to \Gamma(C^{1+n}(E)^*).\]
By \eqref{eq:zero}, and since every element of degree $r$ is a sum of products of generators
of degree $\le r$, the operator $\d_n$ vanishes on elements $x$ of degree $r<n-1$. Hence,
on any element of given degree only finitely many terms of
\[ \d=\sum_{n\ge 0} \d_n\]
act non-trivially. Similarly, we see that all $\d_n$ with $n\neq 1$ must vanish on $C^\infty(M)$. Hence, by the derivation property, the $\d_n$   are $C^\infty(M)$-linear maps except for $n=1$. One defines the \emph{anchor map}
\begin{equation}\label{eq:linftyanchor}
\a\colon E\to TM
\end{equation}
in terms of its action on sections as
\[ \L_{\a(\sigma)}f=\iota(\sigma)\d f;\]
by the above it is non-vanishing only on $E_0$. The restriction of $\d_0$ to sections of weight $1$ is dual to a degree $1$ map $\delta\colon E\to E$,
defining a complex of vector bundles
\begin{equation}\label{eq:complex}
 \cdots \stackrel{\delta}{\lra}
E_{-2}\stackrel{\delta}{\lra}
E_{-1}\stackrel{\delta}{\lra}
E_{0}\lra 0;\end{equation}
one finds that $\a\circ \delta=0$. Finally, one has the $n$-brackets for $n\ge 2$,
\begin{equation}\label{eq:nbrackets}
 [\cdot,\cdot,\ldots,\cdot]_n\colon \Gamma(E)\times \cdots \times \Gamma(E)\to \Gamma(E),
\end{equation}
given by Voronov's \emph{derived bracket} construction  \cite{vor:hig}:
\[ \iota([\sigma_1,\ldots,\sigma_n]_n)= [[\ldots [\d_{n-1},\iota(\sigma_1)],\iota(\sigma_2)],\ldots,\iota(\sigma_n)];\]
alternatively these may be regarded as the components of the co-derivation $Q\colon \Gamma(\mathsf{C}(E))\to \Gamma(\mathsf{C}(E))$ dual to $\d$.
The binary bracket $[\cdot,\cdot]=[\cdot,\cdot]_2$ on $\Gamma(E)$ satisfies the Leibniz rule
\[ [\sigma,f\tau]=f[\sigma,\tau]+(\L_{\a(\sigma)}f)\tau\]
while the higher brackets $[\cdot,\ldots,\cdot]_n$ for $n\ge 3$ are all $C^\infty(M)$-linear. These brackets satisfy a sequence of `higher Jacobi identities', and compatibilities with $\delta$;
see \cite{bon:ca} for a detailed discussion. The anchor vanishes on $n$-brackets for $n\ge 3$, and satisfies $\a([\sigma,\tau])=[\a(\sigma),\a(\tau)]$; in particular, an $L_\infty$-algebroid $E$ determines a singular foliation $\a(\Gamma(E_0))$
on $M$.

Suppose that $F\da E$ is an $L_\infty$-morphism, given by a bundle map $\Phi\colon \mathsf{C}(F)\to \mathsf{C}(E)$, with base map $\varphi\colon N\to M$. Since $\Phi$ is a fiberwise coalgebra morphism, it may be described in terms of
its components $\Phi^n\colon C^{1+n}(F[1])\to E[1],\ n=0,1,\ldots$.  The bundle map $\Phi^0\colon F[1]\to E[1]$ is a morphism of anchored vector bundles and also intertwines the differentials $\delta$. However this map does \emph{not} intertwine brackets, in general:  one usually obtains a `morphism'  only by taking the higher $\Phi^n$ into account. In the special case that $\Phi^n=0$ for $n\neq 0$, we will say that \emph{$\Phi$ has weight zero}.

The pullback map on sections of the dual bundle $\Phi^*\colon \Gamma(\mathsf{C}(E)^*)\to \Gamma(\mathsf{C}(F)^*)$
is similarly determined by a sequence of maps
\[ (\Phi^n)^*\colon \Gamma(E[1]^*)\to \Gamma(C^{1+n}(F)^*)\]
satisfying $(\Phi^n)^*(f\alpha)=(\varphi^*f)\ ((\Phi^n)^*\alpha)$.

\subsection{Basic constructions}
The usual constructions with Lie algebroids generalize to the $L_\infty$-setting.

\subsubsection{$L_\infty$-subalgebroids}
 Suppose that $E$ is an $L_\infty$-algebroid over $M$, and $F\subset E$ is a graded subbundle along $N\subset M$ such that the kernel of the natural pullback map $\mathsf{C}(E)^*\to \mathsf{C}(F)^*$ is $\d$-closed. Then $\mathsf{C}(F)^*$ inherits a differential, making $F$ into an $L_\infty$-algebroid with an obvious weight zero  $L_\infty$-morphism $F\da E$.  We refer to $F$ as an \emph{$L_\infty$-subalgebroid}.

\subsubsection{Products}
For any two $L_\infty$-algebroids $E\to M,\ E'\to M'$ the product $E\times E'\to M\times M'$ acquires the structure of an $L_\infty$-algebroid.
The differential on $\Gamma(\mathsf{C}(E\times E')^*)$ is such that $\d(\alpha\otimes \alpha')=\d\alpha\otimes \alpha'+(-1)^r \alpha\otimes\d\alpha'$ when $\alpha\in \mathsf{C}(E),\ \alpha'\in \mathsf{C}(E')$ have degrees $r,r'$, respectively.
Projections to the two factors define $L_\infty$-morphisms of weight zero,
\[ \pr_E\colon E\times E'\da E,\ \ \ \pr_{E'}\colon E\times E'\da E'.\]

\subsubsection{Restriction to submanifolds}
Let $E\to M$ be an $L_\infty$-algebroid, and $N\subset M$ a submanifold such that $i^!E=\a^{-1}(TN)$ is smooth. Note that $(i^!E)_0=i^!(E_0)$ (the pull-back as an anchored vector bundle), while the summands
$(i^!E)_r$ for $r< 0$ are simply the restrictions $E_r|_N$ as vector bundles.
\begin{lemma}[Pullback to submanifolds] \label{lem:linftypull}
If $i\colon N\hra M$ is a smooth submanifold such that $i^!E=\a^{-1}(TN)$ is smooth, then the subbundle  $i^!E=\a^{-1}(TN)$  inherits a unique $L_\infty$-algebroid structure for which it is an $L_\infty$-subalgebroid.
\end{lemma}
\begin{proof}
We have to show that the kernel of the pull-back map $\Gamma(\mathsf{C}(E)^*)\to \Gamma(\mathsf{C}(i^!E)^*)$ is closed under the differential $\d$. This kernel is the ideal in $\Gamma(\mathsf{C}(E)^*)$ generated by functions $f\in C^\infty(M)$
such that $f|_N=0$, and sections $\alpha\in  \Gamma(E[1]^*)$ such that
$\alpha|_N$ takes values in the annihilator of $i^!E$.  Hence, it suffices that the differentials of these generators are again in the ideal.

If $f\in C^\infty(M)$ with $f|_N=0$, and  $\sigma\in \Gamma(E)$ with
$\a(\sigma)$ tangent to $N$, then $\iota(\sigma)(\d f)|_N=(\L_{\a(\sigma)}f)|_N=0$. This shows that $(\d f)|_N$ is a section of the
annihilator.

If $\alpha\in \Gamma(E[1]^*)$
restricts to a section of the annihilator,
and $\sigma,\tau\in \Gamma(E)$ are two homogeneous sections whose image under the anchor is tangent to $N$, then
\[
\iota(\sigma)\iota(\tau)\d\alpha=-
\iota([\sigma,\tau])\alpha
\pm \iota(\sigma)\d\iota(\tau)\alpha\pm\iota(\tau)\d\iota(\sigma)\alpha,
\]
where all three terms
vanish along $N$. (Here, we are using that $\a([\sigma,\tau])=[\a(\sigma),\a(\tau)]$ is again tangent to $N$.)
This shows that $\d\alpha$ is in the ideal.
\end{proof}

\begin{remark}
On the dual side, the degree $1$ differential $\delta$ on the complex $E$ restricts to the differential on  $i^!E$, and similarly for the anchor.
The 2-bracket on $\Gamma(E)$ induces the 2-bracket on
$\Gamma(i^!E)$, in such a way that $[\sigma|_N,\tau|_N]=[\sigma,\tau]|_N$
whenever $\sigma,\tau$ restrict to sections of $i^!E$. Likewise, all the higher brackets $[\cdot,\ldots,\cdot]_n$ of $E$, with $n\ge 3$ `restrict' to the brackets for $i^!E$.
\end{remark}

\subsubsection{Pullbacks}\label{subsec:pullLinfty}
We define the \emph{pullback} of an $L_\infty$-algebroid $E\to M$, under a smooth map $\varphi\colon N\to M$ that is transverse to the anchor,
as the fiber product
\[ \varphi^!E=E\times_{TM} TN.\]
In the case of an embedding, the $L_\infty$-structure is defined by Lemma \ref{lem:linftypull}; in general
it is defined by the identification of $\varphi^!E$ with  $i^!(E\times TN)$, where $i\colon N\cong \on{Gr}(\varphi)\to M\times N$ is the inclusion as the graph.
We have that $(\varphi^!E)_0=\varphi^!(E_0)$ (the pull-back as an anchored vector bundle), while $(\varphi^!E)_r$ for $r\neq 0$ is just the usual pull-back of $E_r$ as a vector bundle. The pull-back $L_\infty$-algebroid comes with a natural  $L_\infty$-morphism of weight zero,
\begin{equation}\label{eq:!Cmap}
\varphi_!\colon  \varphi^!E\da E,
\end{equation}
given by the inclusion $\varphi^!E\hra E\times TN$ as a sub-$L_\infty$-algebroid, followed by the projection to $E$.

If $N=M\times Q$, with $\varphi$ projection to the first factor, we have that $\varphi^!E\cong E \times TQ$; this
also gives a local description of pullbacks under submersions.

\subsubsection{Automorphisms, derivations}\label{subsec:auto}
For a $\Z$-graded vector bundle $V=\oplus_r V_r\to M$ with finite-rank graded pieces, we define the degree $0$ automorphisms $\wt{\phi}\in\on{Aut}_{\ca{VB}}(V)$ as a collection of
vector bundle automorphisms $\wt{\phi}_r\in \on{Aut}_{\ca{VB}}(V_r)$, all with the same base map  $\phi\colon M\to M$. Similarly, we define a degree $0$ infinitesimal vector bundle automorphism
$\wt{X}\in\mf{aut}_{\ca{VB}}(V)$ as a collection of vector fields $\wt{X}_r\in\mf{aut}_{\ca{VB}}(V_r)$ on the graded pieces, all with the same base vector field $X\in\mf{X}(M)$. As remarked in Section \ref{sec:euler}, if $X$ is complete then it is automatic that all the $\wt{X}_r$ are complete; in this case $\wt{X}$ determines a 1-parameter family of automorphisms $\phi^{\wt{X}}_s\in \on{Aut}_{\ca{VB}}(V)$.
Any infinitesimal automorphism $\wt{X}$ defines a sequence of linear maps $D\colon \Gamma(V_r)\to \Gamma(V_r)$, satisfying the Leibniz rule \eqref{eq:leibnitz} for the vector field $X$.
Alternatively, we can describe it as a degree zero linear map on sections of the graded dual $V^*=\oplus_r (V_r)^*$,
\[ D\colon   \Gamma(V^*)\to \Gamma(V^*),\]
again satisfying the Leibniz rule with respect to $X$. The two $D$'s are related by
\[
X\l\alpha,\tau\r=\l D\alpha,\tau\r +\l \alpha,D\tau\r
\]
for all $\alpha\in\Gamma(V^*)$ and $\tau\in\Gamma(V)$.

 Let $E\to M$ be an $L_\infty$-algebroid. An \emph{infinitesimal $L_\infty$-automorphism} of $E$ is given by an infinitesimal automorphism $\wt{X}$ of the graded vector bundle $\mathsf{C}(E)$,
with base vector field $X$, such that the corresponding (local) flow is by (local) $L_\infty$-automorphisms. Using the discussion above, we obtain an equivalent description in terms of operators on the complex $\Gamma(\mathsf{C}(E)^*)$:

\begin{lemma}
An infinitesimal $L_\infty$-automorphism of an $L_\infty$-algebroid $E$ is equivalent
to a degree 0 derivation $D\colon   \Gamma(\mathsf{C}(E)^*)\to \Gamma(\mathsf{C}(E)^*)$
which commutes with $\d$.

%
\end{lemma}
 %

Note that the restriction of the derivation $D$ to the weight zero summand $\Gamma(\mathsf{C}^0(E)^*)=C^\infty(M)$ is the base  vector field $X$
of the resulting infinitesimal automorphism $\wt{X}$ of $\mathsf{C}(E)$. If $X$ is complete then $\wt{X}$  integrates to a flow $\varphi_{s}^{\wt{X}}\colon\mathsf{C}(E)\to \mathsf{C}(E)$, defining a 1-parameter family of
$L_\infty$-algebroid automorphisms, denoted by the same letter
$\varphi_{s}^{\wt{X}}\colon E \da E$.
We will mainly be interested in the following situation. Let $\sigma\in \Gamma(E_0)= \Gamma(E[1]_{-1})$. Then the interior multiplication operator $\iota(\sigma)$ is a graded derivation of degree $-1$ of $\Gamma(\mathsf{C}(E)^*)$, and its graded commutator with $\d$ is a degree $0$ derivation
\[ \L(\sigma)=[\d,\iota(\sigma)].\]
This satisfies
\[ [\d,\,\L(\sigma)]=[\d,[\d,\iota(\sigma)]]=\hh [[\d,\d],\iota(\sigma)]=0,\]
since $[\d,\d]=2\d\circ \d=0$. On functions,
\[ \L(\sigma)f=\iota(\sigma)\d f=\L_{\a(\sigma)}f.\]
It follows that  $\L(\sigma)$ defines an infinitesimal $L_\infty$-automorphism $\wt{X}$ of
$E$, with base vector field $X=\a(\sigma)$. If $X$ is complete, then this derivation integrates to a 1-parameter family of $L_\infty$-automorphisms.

\subsection{The splitting theorem}
Let $E\to M$ be an $L_\infty$-algebroid, and  $N\subset M$ a submanifold transverse to the anchor of $E$. Let $i\colon N\to M$ be the inclusion, and
$p\colon \nu(M,N)\to N$ the normal bundle; we take the $L_\infty$-algebroid
\[ \nu(E)=p^! i^! E\]
to be the \emph{linear approximation} of $E$ along $N$.  The following result says that
on a tubular neighborhood of $N$, the $L_\infty$-algebroid $E$ is $L_\infty$-isomorphic to its linear approximation.

\begin{theorem}[Splitting theorem for  $L_\infty$-algebroids]\label{th:split-linfinity}
Let $E\to M$ be an $L_\infty$-algebroid, and $N\subset M$ a submanifold
transverse to the anchor of $E$. Then there exists an Euler-like section $\sigma\in \Gamma(E_0)$ with respect to $N$, and any choice of such a section determines an $L_\infty$-isomorphism
\[ \nu(E)\da E|_U,\]
with base map the tubular neighborhood embedding $\psi\colon \nu(M,N)\to U\subset M$ defined by $\a(\sigma)$.
\end{theorem}
\begin{proof}
The existence of Euler-like sections follows from Lemma~\ref{EL-ex}. Since $\a\colon E\to TM$ is tranverse to $N$, it is transverse to the map $\kappa\colon \D(M,N)\to M$. We hence obtain an $L_\infty$-algebroid
\[ \D(E)=\kappa^!E\to \D(M,N)\]
over the deformation space. By our discussion for anchored vector bundles in Section \ref{sec:splitla}, we have the section
\[   w=\f{1}{t}(\D(\sigma)+\theta)\in \Gamma(E_0)\subset \Gamma(E),\]
with $\a(w)=W$. As discussed in Section~\ref{subsec:auto}, this section defines
a degree $0$ derivation
\[ \L(w)\colon \mathsf{C}(\D(E)^*)\to \mathsf{C}(\D(E)^*),\]
commuting with $\d$, and a corresponding infinitesimal $L_\infty$-automorphism of $\D(E)$. The vector field $W$ is complete over $\D(U,N)\subset \D(M,N)$,
and the lifted flow
\[
\wt{\varphi}_s\colon \mathsf{C}(\D(E)|_{\D(U,N)})\to \mathsf{C}(\D(E)|_{\D(U,N)}).
\]
of the vector field $\wt{W}$ on the coalgebra bundle defines a 1-parameter family of $L_\infty$-automorphisms
\[\wt{\varphi}_s\colon  \D(E)|_{\D(U,N)}\da \D(E)|_{\D(U,N)},\]
with base map the flow $\varphi_s$ of $W$.
The map $\varphi_{-1}$ takes $\pi^{-1}(0)$ to $\pi^{-1}(1)$, hence
it induces an $L_\infty$-isomorphism
\[ \wt{\psi}\colon j_0^!\D(E)|_{\D(U,N)}\da  j_1^!\D(E)|_{\D(U,N)}\]
such that $ \wt{\varphi}_{-1}\circ (j_0)_!=(j_1)_!\circ \wt{\psi}$
(cf.~ \eqref{eq:!Cmap}). Using
the identifications $j_0^!\D(E)|_{\D(U,N)}\cong \nu(E)$ and $j_1^!\D(E)|_{\D(U,N)}\cong E|_U$
this is the desired $L_\infty$-isomorphism $\wt{\psi}\colon \nu(E)\da E|_U$.
\end{proof}

Since $\kappa\circ j_1=\on{id}$, and writing $j=j_0$, we may alternatively describe $\wt{\psi}$
as a composition of $L_\infty$-morphisms
\begin{equation}\label{eq:LAsplit1} \xymatrix@C=8ex{ \nu(E)\ar@{-->}[r]^{j_!}\ar[d] & \D(E)|_{\D(U,N)} \ar@{-->}[r]^{\wt{\varphi}_{-1}}\ar[d] & \D(E)|_{\D(U,N)}\ar@{-->}[r]^{\kappa_!}\ar[d]
& E|_U \ar[d]\\
\nu(M,N) \ar[r]_j  & \D(U,N) \ar[r]_{\varphi_{-1}} & \D(U,N) \ar[r]_\kappa & U.
}
\end{equation}
%
%


\begin{remark}
By the result for anchored vector bundles (cf.~ Section \ref{sec:splitla}),
the choice of the Euler-like section $\sigma$ determines a degree $0$  isomorphism of graded vector bundles $\nu(E)\to  E|_U$,
with base map the tubular neighborhood embedding $\psi\colon\nu(M,N)\to U\subset M$ determined by $\a(\sigma)$. (While $E$ has infinite rank, it suffices to apply the result from Section \ref{sec:splitla} to its finite-rank approximations $\oplus_{r=-k}^0 E_r$.) The construction of this map depends only on the 2-bracket $[\cdot,\cdot]$. In general, this map $\nu(E)\to E|_U$ determines only the weight $0$ component of the morphism
$\wt{\psi}\colon \nu(E)\da E|_U$
 of the theorem. The construction of the full morphism involves all the higher $n$-brackets as well.
\end{remark}

\begin{corollary}[Local splitting theorem for $L_\infty$-algebroids]
Let $E\to M$  be an $L_\infty$-algebroid, $S\subset M$ a leaf of the corresponding singular foliation, and $m\in S$. Choose a submanifold $i\colon N\hra M$ containing $m$, such that $T_mM=T_mN\oplus T_mS$. Replacing $N$ with a smaller submanifold if needed, the $L_\infty$-algebroid $E$ is locally $L_\infty$-isomorphic, near $m$, to the product $i^!E\times TS$.
\end{corollary}
\begin{proof}
Replacing $N$ with a smaller submanifold if needed, we ensure that $N$ is transverse to the foliation, and that $\nu(M,N)$ is trivial:  $\nu(M,N)=N\times \mathsf{P}$, where $\mathsf{P}=T_mS$. But then
$p^!i^!E=i^!E\times T\mathsf{P}$ (see Section~\ref{subsec:pullLinfty}), so the claim follows from the theorem.
\end{proof}
As another direct application, we obtain a local description of transitive $L_\infty$-algebroids.

\begin{corollary}[Local trivializations of transitive $L_\infty$-algebroids]
Let $E\to M$ be a transitive $L_\infty$-algebroid, i.e., such that the anchor map is surjective. Given $m\in M$ and a tubular neighborhood embedding $\psi\colon T_mM\to M$, with image $U$, there is an $L_\infty$-isomorphism
\[ E|_U\da TU\times \g\]
where $\g$ is an $L_\infty$-algebra.
\end{corollary}
\begin{proof}
We apply the theorem to $i\colon N=\{m\}\hra M$. Thus $\g=i^!E$ is an $L_\infty$-algebra, and
$\nu(E)=p^!i^!E=T_mM\times \g\cong TU\times \g$.
Let $X\in\mf{X}(U)$ be the image of the Euler vector field $\E$ under $\psi$, and
$\sigma\in \Gamma(E|_U)$ its image under any choice of splitting of
$\a\colon E_0\to TM$. Then Theorem \ref{th:split-linfinity} specializes to an
$L_\infty$-isomorphism $\wt{\psi}\colon \nu(E)\da E|_U$.
\end{proof}


\section{The splitting theorem for Courant algebroids}\label{sec:splitcourant}

We now discuss a splitting theorem for Courant algebroids; its proof follows the same strategy as for the anchored vector bundles and $L_\infty$-algebroids.
Basic references on Courant algebroids include \cite{lib:cou,liu:ma,roy:co,sev:let,sev:poi,sev:poi1}.

\subsection{Basic definitions and constructions}\label{subsec:basic}
Throughout, we will take `\emph{metric}' to be synonymous with `non-degenerate symmetric bilinear form'.  A \emph{Courant algebroid}
is a vector bundle $\AA$ over a manifold $M$, equipped with a metric $\l\cdot,\cdot\r$ on its fibers, a bundle map $\a\colon \AA\to TM$ over the identity map on $M$
(called the {\em anchor} map), and an $\mathbb{R}$-bilinear bracket
$\Cour{\cdot,\cdot}$ on its space of sections, such that the following
properties hold, for all sections $\sigma,\tau_1,\tau_2\in\Gamma(\AA)$:
\begin{enumerate}
\item $\a(\sigma)\l\tau_1,\tau_2\r=\l\Cour{\sigma,\tau_1},\tau_2\r+\l\tau_1,\Cour{\sigma,\tau_2}\r$,
\item $\Cour{\sigma,\Cour{\tau_1,\tau_2}}=\Cour{\Cour{\sigma,\tau_1},\tau_2}+\Cour{\tau_1,\Cour{\sigma,\tau_2}}$,
\item $\Cour{\tau_1,\tau_2}+\Cour{\tau_2,\tau_1}=\a^* \d\ \l\tau_1,\tau_2\r$.
\end{enumerate}
In the last line, the metric is used to identify $\AA^*\cong \AA$, hence $\a^*\colon T^*M\to \AA$.  One refers to $\Cour{\cdot,\cdot}$ as the \emph{Courant bracket} (or sometimes the \emph{Courant-Dorfman bracket}). 
From the axioms, one can derive additional properties, such as the derivation property
\begin{equation}\label{eq:dercou}
\Cour{\sigma,\ f\tau}=f\Cour{\sigma,\tau}+(\a(\sigma)f)\tau,
\end{equation}
for all $f\in C^\infty(M)$ and $\sigma,\tau\in\Gamma(\AA)$, and the
property $\a\circ \a^*=0$. (See \cite{uch:rem}.)
The
anchor map satisfies
$\a(\Cour{\sigma,\tau})=[\a(\sigma),\a(\tau)]$,
%
which implies that $(\AA,\a)$ is an \emph{involutive} anchored vector bundle (in the sense of Section~\ref{sec:splitla}). As a result,  one obtains:
\begin{proposition}
For any Courant algebroid $\AA$ over $M$, the image of the anchor map
defines a singular foliation $\J=\a(\Gamma(\AA))$.
\end{proposition}
%

A Courant algebroid is called \emph{transitive} if the anchor map is surjective. Clearly, the restriction $\AA|_\O$ of a Courant algebroid $\AA$ to any leaf $\O$ of its singular foliation is a transitive Courant algebroid, and $\AA$ is the union over these restrictions.

A first example is the \emph{standard Courant algebroid} $\T M=TM\oplus T^*M$, with anchor given by the projection to the first summand,
the bilinear form $\l\cdot,\cdot\r$ extending the pairing between $TM$ and $T^*M$ (thus, $TM$ and $T^*M$ are Lagrangian), and with the Courant bracket on sections given by
\begin{equation}\label{eq:courantbracket}
 \Cour{X_1+\alpha_1,X_2+\alpha_2}=[X_1,X_2]+\ca{L}_{X_1}\alpha_2-\iota_{X_2}\d\alpha_1,
 \end{equation}
for vector fields $X_1,X_2$ and 1-forms $\alpha_1,\alpha_2$.
More generally, given a closed 3-form $\eta \in \Omega^3(M)$, one has the $\eta$-twisted Courant algebroid $\T M^{\eta}$, with the same pairing and anchor as before,
 and Courant bracket obtained by adding a term $\iota_{X_1}\iota_{X_2}\eta$ to the right hand side of \eqref{eq:courantbracket}.

A Courant algebroid over a point is the same as a \emph{metrized} Lie algebra $\dd$ (a Lie algebra with an $\ad$-invariant metric).  More generally, given a Lie algebra action of $\dd$ on a manifold $M$, with coisotropic stabilizers,  the trivial bundle
$\AA=M\times \dd$ has the structure of an  \emph{action Courant algebroid}  \cite{lib:cou}. Here the anchor $\a$ is given by the action, and the metric and Courant bracket on sections of $\AA$ extend the metric and Lie bracket on $\dd$, regarded as constant sections.

A subbundle $E\subset \AA$ of a Courant algebroid is called a \emph{Dirac structure} if it is Lagrangian (that is, $E=E^\perp$) and its space of sections is closed under the Courant bracket. In this case, $E$ becomes a Lie algebroid, with the Lie bracket and anchor obtained by restriction of the Courant bracket and anchor on $\AA$. Note that
the anchor of a Dirac structure defines another singular foliation of $M$, whose leaves are contained in those of $\AA$. Graphs of Poisson structures and graphs of closed 2-forms define Dirac structures inside the standard Courant algebroid $\T M$.
Given a metrized Lie algebra $\dd$ with an action on $M$ with coisotropic stabilizers, any Lagrangian Lie subalgebra defines a Dirac structure $M\times \g$ inside the action Courant algebroid $M\times \dd$.


\subsection{Courant algebroid automorphisms}\label{subsec:caauto}
We denote by $\mf{aut}_{\ca{CA}}(\AA)$ the infinitesimal Courant algebroid automorphisms, given by vector fields $\wt{X}\in\mf{X}(\AA)$ on the total space $\AA$, with base vector field $X\in\mf{X}(M)$, whose local flows are by Courant algebroid automorphisms (i.e., preserving the vector bundle structure, anchor, bracket, and metric).  This is expressed as compatibility of the corresponding linear operator $D$ on
$\Gamma(\AA)$ (cf. Section~ \ref{subsec:convent}) with the anchor, metric and bracket:
\[ \begin{split}
\a(D\tau)&=[X,\a(\tau)],\\
\ca{L}_X \l \tau_1,\tau_2\r&=\l D\tau_1,\tau_2\r+\l\tau_1,D\tau_2\r,\\
D\Cour{\tau_1,\tau_2}&=\Cour{D\tau_1,\tau_2}+\Cour{\tau_1,D\tau_2}.
\end{split}\]
In particular, for every
section $\sigma\in \Gamma(\AA)$ the operator
$
D=\Cour{\sigma,\cdot}
$
defines an `inner' infinitesimal Courant automorphism
\begin{equation}\label{eq:alift}
\wt{\a}(\sigma)\in \aut_{\ca{CA}}(\AA)\end{equation}
lifting the vector field $\a(\sigma)$. Given a closed 2-form $\omega\in\Omega^2(M)$, with associated
bundle map $\omega^\flat\colon TM\to T^*M$, one finds that
\[D_\omega=\a^*\circ \omega^\flat\circ \a\colon \Gamma(\AA)\to \Gamma(\AA^*)\cong \Gamma(\AA)\]
defines an infinitesimal Courant automorphism, with $X=0$. If $\omega$ is exact, then this
infinitesimal Courant automorphism is inner:
\begin{lemma}\label{lem:alpha}
For any 1-form $\alpha\in \Omega^1(M)$,
\begin{equation}\label{eq:identity}
 D_{d\alpha}=- \Cour{\a^*\alpha, \cdot}.\end{equation}
 In particular, if $\alpha$ is closed then $\Cour{\a^*\alpha, \cdot}=0$.
 \end{lemma}
 \begin{proof}
We calculate, for any  $\gamma,\zeta\in\Gamma(\AA)$:
\begin{align*}
\l \Cour{\a^*\alpha, \gamma},\zeta\r&=\l \a^*\d \l\a^*\alpha,\gamma\r-\Cour{\gamma,\a^*\alpha},\zeta\r\\
&=\l \d\l\alpha,\a(\gamma)\r,\a(\zeta)\r -\L_{\a(\gamma)}\l \a^*\alpha,\zeta\r
+\l \a^*\alpha,\Cour{\gamma,\zeta}\r\\
&=\L_{\a(\zeta)}\l \alpha,\a(\gamma)\r-\L_{\a(\gamma)}\l\alpha,\a(\zeta)\r+\l\alpha,[\a(\gamma),\a(\zeta)]\r\\
&=-(\d\alpha)(\a(\gamma),\a(\zeta))\\
&=-\l D_{d\alpha}(\gamma),\zeta\r.\qedhere
\end{align*}
\end{proof}
We will use the notation
$$
R_\omega=\exp(D_\omega)=\on{id}+D_\omega
$$
for the global Courant automorphism defined by a closed 2-form $\omega$, and refer to it as a {\em gauge transformation}.

\begin{proposition}\label{prop:gaugerelated}
Let $\sigma,\tau\in \Gamma(\AA)$, with $\tau-\sigma = \a^*\beta$ for some $\beta\in \Omega^1(M)$, and let $X=\a(\tau)=\a(\sigma)$. Then the (local) flows $\wt{\psi}_s,\wt{\phi}_s$ on $\AA$ defined by
$\wt{\a}(\tau),\wt{\a}(\sigma)$ are related by a family of
gauge transformations:
\[ \wt{\psi}_s=R_{\varpi_s}\circ \wt{\phi}_s,\ \ \  \varpi_s=  - \d \int_0^s (\phi_{u})_*\beta\ \d u,\]
where $\phi_s$ is the flow of $X$.
\end{proposition}
\begin{proof}
Recalling the action of automorphisms of $\AA$ on $\Gamma(\AA)$ (see Section~\ref{subsec:convent}), which we denote by a dot,
we will show that
\[ \wt{\psi}_s\circ \wt{\phi}_s^{-1}.\gamma=R_{\varpi_s}(\gamma)
\]
for all sections $\gamma\in \Gamma(\AA)$.
Since the two sides coincide for $s=0$, it suffices to compare their $s$-derivatives:
\begin{align*}
 \f{d}{d s}(\wt{\psi}_s\circ \wt{\phi}_s^{-1}).\gamma&=
\wt{\psi}_s\circ \Cour{\tau-\sigma,\cdot}\circ \wt{\phi}_s^{-1}.\gamma\\
&= \wt{\psi}_s\circ \Cour{\a^*\beta,\cdot}\circ \wt{\phi}_s^{-1}.\gamma\\
&=\wt{\psi}_s\circ D_{-d\beta}\circ \wt{\phi}_s^{-1}.\gamma\\
&= D_{-d(\phi_s)_*\beta}\circ \wt{\psi}_s\circ\wt{\phi}_s^{-1}.\gamma\\
&=D_{-d(\phi_s)_*\beta}(\gamma)\\
&=\f{d}{d s} R_{\varpi_s}.\gamma,
\end{align*}
where the second-to-last last step used that $\wt{\phi_s}$ and $\wt{\psi}_s$ have the same base map $\phi_s$, and hence $\a(\wt{\psi}_{s}\circ \wt{\phi}_{s}^{-1}.\gamma)=\a(\gamma)$.
\end{proof}
Note that the formula can also be written as
\begin{equation}\label{eq:altformula}
 \wt{\psi}_s=\wt{\phi}_s \circ R_{\omega_s},\ \ \ \omega_s= - \d \int_0^s (\phi_{u})^*\beta\ \d u.\end{equation}
This follows from $R_{\varpi_s}\circ \wt{\phi}_s=\wt{\phi}_s\circ R_{\phi_s^*\varpi_s}$, since
$\phi_s^*\varpi_s=\omega_s$.

\subsection{Pullbacks of Courant algebroids}\label{subsec:pull}

Given a Courant algebroid $\AA$ over $M$ and a map $\varphi\colon N\to M$ transverse to the anchor of $\a\colon \AA\to TM$, we define the {\em pull-back Courant algebroid} $\varphi^!\AA$ over $N$ as follows (see \cite[Section 2.2]{lib:cou}). Let
\begin{equation}\label{eq:C}
C\subset \AA\times \T N
\end{equation}
be the fiber product with respect to the anchor $\a \colon \AA\to TM$ and the map $T\varphi\circ \a_{\T N}\colon \T N\to TM$.
Then $C$ is a coisotropic subbundle along $\on{Gr}(\varphi)\subset M\times N$, and $C^\perp$
 lies in the kernel of the anchor. It follows that the quotient
 \[
 \varphi^!\AA=C/C^\perp
 \]
inherits a Courant algebroid structure over $N\cong \on{Gr}(\varphi)$, with the Courant bracket induced from that on sections of $\AA\times \T N$. (Note that the pull-back $\varphi^!\AA$ as a Courant algebroid is different from the pull-back as an anchored vector bundle.)

The diagonal inclusion of $C$ in
$\AA\times \ol{\varphi^!\AA}$ is a Lagrangian subbundle along the graph of $\varphi$, and
defines a {\em Courant morphism} (in the sense of \cite[Sec.~2.2]{bur:cou}),
\[ \varphi_! \colon \varphi^!\AA\da \AA,
\]
see \cite[Proposition 2.10]{lib:cou}.
Here $\ol{\varphi^!\AA}$ denotes the Courant algebroid with the same bracket and anchor as $\varphi^!\AA$ but with minus the metric $-\l\cdot,\cdot\r$. We have that $(\varphi\circ \varphi')^!\AA
=(\varphi')^!\,\varphi^!\,\AA$ under composition of maps, provided that the appropriate transversality conditions are satisfied.

Given a Dirac structure $E\subset \AA$ such that $\varphi$ is transverse to the anchor of $E$, one obtains a pull-back Dirac structure $\varphi^!E\subset \varphi^!\AA$ as the image of  $C\cap (E\times TN)$ under the quotient map $C\to C/C^\perp=\varphi^!E$. (Its Lie algebroid structure coincides with that as the pull-back Lie algebroid.)

Special cases of the pull-back of Courant algebroids include:
\begin{enumerate}
\item $\varphi^!\T M=\T N$. (See \cite[Proposition 2.9]{lib:cou}.)
\item If $\varphi\colon M\times Q\to M$ is the projection onto the first factor, then $\varphi^!\AA
=\AA\times \T Q$.
\item If $\varphi$ is a diffeomorphism, then $\varphi^!\AA$ is the usual pull-back
$\varphi^*\AA$ as a vector bundle.
\end{enumerate}
Given a section $\tau\in \Gamma(\AA)$ and a vector field $Z\in \mf{X}(N)$ such that
$Z\sim_{\varphi}\a(\tau)$, we obtain a section of
$\varphi^!\AA$ by restricting $\tau\times Z\in \Gamma(\AA\times \T N)$ to a section of $C\to \on{Gr}(\varphi)\cong N$, followed by the quotient map to $C/C^\perp$. We will denote this section by
\begin{equation}\label{eq:awfulnotation}
\tau\times_\varphi Z\in \Gamma(\varphi^!\AA);
\end{equation}
its image under the anchor is $Z$. Given another map $\varphi'\colon N'\to N$, transverse to the anchor of $\varphi^!\AA$, and a vector field $Z'\in\mf{X}(N')$ such that $Z'\sim_{\varphi'}Z$, we have that
\[ \tau\times_{\varphi\circ \varphi'}Z'=(\tau\times_\varphi Z)\times_{\varphi'}Z'.\]



\subsection{The splitting theorem}
 Suppose that $\AA\to M$ is a Courant algebroid whose anchor $\a$ is transverse to a given submanifold $i\colon N\hra M$.
 The description of the pull-back Courant algebroid simplifies \cite[Proposition 2.8]{lib:cou}:
 \begin{equation} i^!\AA=\a^{-1}(TN)/\a^{-1}(TN)^\perp.\end{equation}
Any section $\tau\in \Gamma(\AA)$ such that
 $\a(\tau)$ is tangent to $N$ defines a section
 \begin{equation}\label{eq:sectionpullback}
  i^!\tau\in \Gamma(i^!\AA),\end{equation}
 by taking the image of $\tau|_N\in \Gamma(C)$ under the quotient map; in terms of
 \eqref{eq:awfulnotation} we have that $i^!\tau=\tau\times_i \a(\tau)|_N$.
 As usual, we denote by $p\colon \nu(M,N)\to N$ the projection; we take
\begin{equation}\label{eq:nuA}
\nu(\AA)=p^!i^!\AA \to \nu(M,N)
\end{equation}
to be the linear approximation for $\AA$ around $N$. This bundle has a distinguished \emph{Euler section}
\begin{equation}\label{eq:euler2}
 \epsilon =0\times_p \E \in \Gamma(\nu(\AA)).
 \end{equation}
 By construction,
$\l\epsilon,\epsilon\r=0,\ \epsilon|_N=0$, and  $\a(\epsilon)=\E$. For $\tau\in \Gamma(\AA)$ such that $\a(\tau)$ is tangent to $N$, we have the section
$\nu(\tau)=i^!\tau\times_p \nu(\a(\tau))$. If $\sigma$ is an Euler-like section,
then $\nu(\sigma)=\epsilon$. Sometimes, it is useful to take $\sigma$ to be isotropic:

\begin{lemma}
There exists an Euler-like section
 $\sigma\in\Gamma(\AA)$ which is  \emph{isotropic}, that is,  $\l\sigma,\sigma\r=0$. Given a proper $G$-action on $M$, with a lift to an action by Courant automorphisms of $\AA$, one can take $\sigma$ to be $G$-invariant.
\end{lemma}
\begin{proof}
Choose a subbundle $R\subset TM$ with $R|_N=TN$.
Since $N$ is transverse to the anchor of $\AA$, the subbundle $R$ remains transverse to the
anchor over some neighborhood $U_1\subset M$ of $N$; hence, over this neighborhood, $\a^{-1}(R|_{U_1})\subset \AA|_{U_1}$ is a subbundle. Since $R|_{U_1}$ contains the coisotropic subspace $\ker(\a_m)$ for $m\in U_1$, it is coisotropic. Choose a complementary subbundle $F_1\subset \AA|_{U_1}$;
by Lemma \ref{lem:iso} below this determines a complementary  \emph{isotropic} subbundle $F\subset \AA|_{U_1}$. By construction, the restriction of the anchor to
 $F$ is transverse to $N$. Hence, there exists an Euler-like section $\sigma_1\in \Gamma(F)$ (see Lemma~\ref{EL-ex}).
After multiplying (if necessary) by a `bump function' supported in $U_1$, the section $\sigma_1$ extends by zero to an Euler-like isotropic section $\sigma\in \Gamma(\AA)$. In the $G$-equivariant case, if the action on $M$ is proper, one can take all of the choices to be
$G$-equivariant, and the resulting $\sigma$ is then $G$-invariant.
\end{proof}

The previous proof used the following fact:
\begin{lemma}\label{lem:iso}
Let $V\to M$ be a metrized vector bundle, and $C\to M$ a coisotropic subbundle. Then any choice of a subbundle $F_1\subset V$ complementary to $C$ determines  an isotropic subbundle $F\subset C$ complementary to $C$.
\end{lemma}
\begin{proof}	
This is well-known in case that the metric on $V$ is split, and $C$ is Lagrangian: In this case, one takes $F$ to be the midpoint between $F_1$ and $F_1^\perp$ in the affine space of complementary subspaces.
(See e.g. \cite[Proposition 1.4]{me:clifbook}.) We may reduce to this case, by considering the subbundle
$V'=C^\perp\oplus F_1$. Since
\[ V'\cap (V')^\perp=(C^\perp+ F_1) 	\cap C\cap F_1^\perp
=C^\perp\cap F_1^\perp=(C+F_1)^\perp=V^\perp=0,\]
we see that $V'$ is a metrized subbundle, containing $C'=C^\perp$ as a Lagrangian subbundle, and $F_1$  as a complementary subbundle.
\end{proof}

Given a Dirac structure $E\subset \AA$ such that its anchor $\a|_E$ is transverse to
$N$, we obtain a Dirac structure
\begin{equation}\label{eq:nuE}
\nu(E)=p^!i^!E\subset \nu(\AA).
\end{equation}

\begin{theorem}[Splitting theorem for Courant algebroids]\label{th:split-courant}
 Let $\AA\to M$ be a Courant algebroid, and let $N\subset M$ be a submanifold transverse to the anchor of $\AA$.
  \begin{enumerate}
 \item
 There exists an isotropic Euler-like section $\sigma\in \Gamma(\AA)$, and any choice of an Euler-like section (not necessarily isotropic) determines an isomorphism of Courant algebroids,
 \[ \xymatrix{ \nu(\AA)\ar[r]_{\wt{\psi}} \ar[d] & \AA|_U\ar[d]\\ \nu(M,N)\ar[r]_{\psi} & U
 }
 \]
 where $\psi$  is the tubular neighborhood embedding $\nu(M,N)\to U\subset M$
  defined by $X=\a(\sigma)$.
 \item
 For any section $\tau\in\Gamma(\AA)$ such that $\a(\tau)$ is tangent to $N$ and $\Cour{\sigma,\tau}=0$, we have that $\wt{\psi}^*\tau=\nu(\tau)$, where
 $\wt{\psi}^*\tau=\wt{\psi}^{-1}\circ \tau\circ \psi$. In particular, if $\sigma$ is chosen to be isotropic
then
 \[ \wt{\psi}^*\sigma=\epsilon.\]
 \item If $E\subset \AA$ is a Dirac structure in $\AA$, such that its anchor $\a_E$
 is transverse to $N$, and if $\sigma$ is an Euler-like section of $E$, then  the Courant algebroid isomorphism $\wt{\psi}\colon \nu(\AA)\to \AA|_U$ takes $\nu(E)$ to $E|_U$.
 \end{enumerate}
Given a  proper $G$-action on $M$, with a lift to an action by Courant automorphisms of $\AA$,
then the choice of a $G$-invariant Euler-like section $\sigma$ gives a  $G$-equivariant map
$\wt{\psi}$.
 \end{theorem}
The proof will be given in Section \ref{sec:proof-courant}.\medskip

\begin{remark}
In the case of an exact Courant algebroid $\AA=\T M$ (with the standard Courant bracket, possibly twisted by a closed 3-form $\eta\in \Omega^3(M)$, one has a more precise description of the normal form for Dirac structures. See \cite[Theorem 5.1]{bur:spl}.
\end{remark}

  Note that if the normal bundle to the submanifold $N$ is trivial, then the choice of a trivialization $\nu(M,N)=N\times\P$ simplifies the model for the Courant algebroid $\AA$ in a neighborhood of $N$ to
 \begin{equation} \label{eq:simpl}
 \nu(\AA)=i^!\AA\times \T\P.
 \end{equation}
In the case of a Dirac structure $E\subset \AA$ with $\a_E$ transverse to $N$, the model for $E$ becomes
\begin{equation}\label{eq;simpl1}
 \nu(E)=i^!E\times T\P,
\end{equation}
as a Dirac structure in $\nu(\AA)$. In particular, taking $N$ to be a `small' transversal to a leaf, one obtains a local
 normal form for Courant algebroids, as a product of the standard Courant algebroid over the leaf and a Courant algebroid over the transversal. We will discuss this local splitting theorem in Section \ref{subsec:localsplitting}, but let us already
 point out the following consequence.

 \begin{corollary}[Local trivializations of transitive Courant algebroids]
\label{cor:transitivecour} Let $\AA\to M$ be a transitive Courant algebroid. Given a point $m\in M$ and a tubular neighborhood embedding $\psi\colon T_mM\to M$, with image $U$, there is an isomorphism with the product Courant algebroid
 \[  \AA|_U\cong \T U\times \mathfrak{q}\]
 where $\mathfrak{q}$ is a metrized  Lie algebra. Given a Dirac structure $E\subset \AA$
 which is transitive in the sense that $\a|_E$ is surjective, we may choose this isomorphism
 in such a way that it identifies
 \[ E|_U\cong TU\times\g,\]
 where $\g\subset \mf{q}$ is a Lagrangian Lie subalgebra.
 \end{corollary}
 \begin{proof}
 Let $X\in\mf{X}(U)$ be the image of the Euler vector field $\E$. By Lemma \ref{lem:iso}
 below, we may choose an isotropic splitting of $\AA$, i.e., a right inverse $TM\to \AA$ to the anchor map whose image is an isotropic subbundle. The image of $X$ under the isotropic splitting is an isotropic section $\sigma\in \Gamma(\AA|_U)$ with $\sigma|_m=0$ and $\a(\sigma)=X$.

Since the anchor map of $\AA$ is surjective, the inclusion $i\colon \{m\}\to M$ is transverse to the anchor. Let $\mathfrak{q}=i^!\AA$ (recall that a Courant algebroid over a point is just a metrized  Lie algebra).  Its pull-back $\nu(\AA)=p^!i^!\AA$ to $\nu(M,\{m\})=T_mM$ is the product $\T(T_mM)\times \mathfrak{q}$. Theorem \ref{th:split-courant}, applied to $\AA|_U$,
gives an isomorphism $\wt{\psi}\colon \T(T_mM)\times \mathfrak{q}\to \AA$ with base map $\psi$. Using $\psi$, we may identify $T_mM\cong U$, hence $\T(T_mM)=\T U$.

Given a Dirac structure $E\subset \AA$ such that $\a|_E$ is surjective, we may take
$\sigma$ to be the image of $X$ under a splitting of $\a|_E\colon E\to TM$. Then the isomorphism for $\AA|_U$ restricts to an isomorphism $E_U\to TU\times\g$, where
$\g=\ker(\a_E)|_m$.
\end{proof}

\subsection{Proof of the splitting theorem}\label{sec:proof-courant}
The fact that the inclusion $i\colon N\to M$ is transverse to $\a$ implies that $\kappa\colon \D(M,N)\to M$ is transverse to $\a$. So we can define the Courant algebroid
\[
\D(\AA)=\kappa^!\AA\to \D(M,N).
\]
This satisfies
\[ \D(\AA)|_{M\times\R^\times}=\AA\times \T\R^\times,\ \ \
j^!\D(\AA)=\nu(\AA),
\]
since over $M\times \R^\times$, $\kappa$ is just projection to the first factor, and
since $\kappa\circ j=i\circ p$. Making use of the construction \eqref{eq:awfulnotation}, we define various sections of $\D(\AA)$:
\begin{enumerate}
\item[(i)]
Since $\Theta\sim_\kappa 0$ the section
\[ \theta=0\times_\kappa \Theta\in \Gamma(\D(\AA))\]
is defined; it has the properties $\a(\theta)=\Theta$ as well as
\[ \theta|_{M\times \R^\times}=0\times t\f{\p}{\p t},\ \ \
j^!\theta=-\epsilon.
\]
The second identity follows from $\epsilon=0\times_p \E$ and
$-\E\sim_j \Theta$.
\item[(ii)]
Given a section $\tau\in \Gamma(\AA)$ such that the vector field $Y=\a(\tau)$ is tangent to $N$, we can also define
\[ \D(\tau)=\tau\times_\kappa \D(Y)\in \Gamma(\D(\AA))
\]
since $\D(Y)\sim_\kappa Y=\a(\tau)$. This satisfies
\[ \D(\tau)|_{M\times \R^\times}=\tau\times 0,\ \ \ j^!\D(\tau)=\nu(\tau)\]
where $\nu(\tau)=i^!\tau\times_p \nu(Y)$. Note that the map $\tau\mapsto \D(\tau)$ (hence also $\tau\mapsto \nu(\tau)$) is bracket preserving. Indeed,
$\D(\Cour{\tau_1,\tau_2})=\Cour{\D(\tau_1),\D(\tau_2)}$  is obvious over $M\times\R^\times$, hence holds globally by continuity.
\item[(iii)] \label{it:4c}
Let $\sigma\in \Gamma(\AA)$ be an Euler-like isotropic section. Since $j^!(\D(\sigma)+\theta)=\epsilon-\epsilon=0$, the section $\D(\sigma)+\theta$ vanishes along the hypersurface $\nu(M,N)$. This shows that $\D(\sigma)+\theta$ is divisible by $t$, defining a section
\[ w\in\Gamma(\D(\AA))\]
such that $\a(w)=W$ and
\[ w|_{M\times \R^\times}=\f{1}{t}\sigma+\f{\p}{\p t}\in \Gamma(\AA\times \T\R^\times).\]
If $\tau\in \Gamma(\AA)$ is a section such that $\a(\tau)$ is tangent to $N$ and $\Cour{\sigma,\tau}=0$, then $\Cour{w,\D(\tau)}=0$. (This is obvious over $M\times\R^\times$, hence holds true globally by continuity.)
 In particular, this hypothesis is true for $\tau=\sigma$, since $\sigma$ is isotropic:  $\Cour{\sigma,\sigma}=\a^*\d\l\sigma,\sigma\r=0$. Therefore,
 \[ \Cour{w,\D(\sigma)}=0.\]
\end{enumerate}

With these preparations, the proof of the splitting theorem for Courant algebroids proceeds parallel to that for involutive anchored vector bundles:
\begin{proof}[Proof of Theorem \ref{th:split-courant}]
\begin{enumerate}
\item Let $\wt{\a}(w)\in\mf{aut}_{\ca{CA}}(\D(\AA))$ be the infinitesimal Courant automorphism defined by $w$. On $\D(\AA)|_{\D(U,N)}$, its flow is complete (see Section~\ref{subsec:convent}), and gives a
1-parameter family of Courant automorphisms $\wt{\varphi}_s\in\Aut(\D(\AA)|_{\D(U,N)})$, covering the flow $\varphi_s$ of $W$ on the base. Putting $s=-1$, since $\varphi_{-1}$ takes $\pi^{-1}(0)$ to $\pi^{-1}(1)$,
it induces a Courant algebroid isomorphism
\[ \wt{\psi}\colon j_0^!\D(\AA)|_{\D(U,N)}\to j_1^!\D(\AA)|_{\D(U,N)}.\]
After identification $\nu(\AA)=j_0^!(\D(\AA)|_{\D(U,N)})$ and $\AA|_U=j_1^!\D(\AA)|_{\D(U,N)}$,
this is the desired Courant isomorphism
$\wt{\psi}\colon \nu(\AA)\to \AA|_U$, with base map
$\psi$.

\item By the discussion at the beginning of this section, the condition $\Cour{\sigma,\tau}=0$
(which holds true in particular for $\tau=\sigma$, since $\sigma$ is isotropic)
implies that $\Cour{w,\D(\tau)}=0$. Hence, $\D(\tau)$ is invariant under the flow $\wt{\varphi}_s$. Since $j_1^!\D(\tau)=\tau$ while $j_0^!\D(\tau)=\nu(\tau)$, it follows that
$\wt{\psi}$ takes $\nu(\tau)$ to $\tau$.

\item Suppose that $E\subset \AA$ is a Dirac structure, with $\a|_E$ transverse to
$N$, and that the section $\sigma$ takes values in $E$. Then $\D(\sigma)$ is a section of
$\D(E)$. The section $\theta\in \Gamma(\D(\AA))$ takes values in $\D(E)$ as well:
This is clear over $M\times\R^\times$ since $\D(E)|_{M\times \R^\times}=E\times T\R^\times$, while $\theta$ restricts to $0\times t\f{\p}{\p t}$. It follows that $w\in\Gamma(\D(E))\subset \Gamma(\D(\AA))$. Since $\D(E)$ is a Dirac structure,
the vector field $\wt{a}(w)$ is tangent to $\D(E)$, hence its flow $\wt{\varphi}_s$
preserves $\D(E)$. Hence, the Courant algebroid isomorphism $\wt{\varphi}_{-1}$
preserves the Dirac structure, and hence the induced Courant isomorphism $\wt{\psi}$ takes
$j_0^! \D(E)|_{\D(U,N)}=\nu(E)$ to $j_1^! \D(E)|_{\D(U,N)}=E|_U$.
\end{enumerate}
\end{proof}
\begin{remark}
Using the notion of \emph{Courant morphisms} (see e.g. \cite{bur:cou,lib:cou}), one can write $\wt{\psi}$ as a composition
\[
\wt{\psi}=\kappa_!\circ \wt{\varphi}_{-1}\circ j_!,
\]
just as in the case of anchored vector bundles and $L_\infty$-algebroids, see \eqref{eq:LAsplit12} and \eqref{eq:LAsplit1}.
\end{remark}

\subsection{Changing $\sigma$}
The construction of the Courant isomorphism $\wt{\psi}\colon \nu(\AA)\to \AA|_U$ depends on the choice of an Euler-like section $\sigma$. Given a 1-form $\alpha\in\Omega^1(M)$
with $\alpha|_N=0$, we obtain a new section
\[ \sigma'=\sigma + \a^*\alpha\]
which is again Euler-like, with $\a(\sigma)=\a(\sigma')=X$.

\begin{proposition}\label{prop:gauge}
The Courant morphisms $\wt{\psi}',\wt{\psi}$ determined by $\sigma',\sigma$ differ by a gauge transformation:
$\wt{\psi}'=\wt{\psi}\circ R_\omega$, where $\omega\in \Omega^2(\nu(M,N)$ is the exact 2-form
\[ \omega= \d \int_0^1 \f{1}{u} m_u^* \psi^*\alpha\ \d u.\]
Here $m_t\colon \nu(M,N)\to \nu(M,N)$ is scalar multiplication by $t$.
\end{proposition}
\begin{proof}
Since $\alpha|_N=0$
the form $\f{1}{t}\alpha\in \Omega^1(M\times\R^\times)$ extends to a form $\D(\alpha)$ on the deformation space, with $t \D(\alpha)=\kappa^*\alpha$.  On the other hand,
$\sigma'$ defines a section $w'$ of $\D(\AA)$, given over $M\times\R^\times$ by $w'=\f{1}{t}\sigma'+\f{\p}{\p t}$. We have
\[ w'=w + \a^*(\D(\alpha)),\]
since this identity holds true over $M\times \R^\times$. Let $\wt{\varphi}_s,\ \wt{\varphi}'_s$ be the
families of Courant automorphisms defined by $w,w'$. By Proposition \ref{prop:gaugerelated} (see \eqref{eq:altformula}), we have that
\begin{equation}\label{eq:varpir}
\wt{\varphi}'_s=\wt{\varphi}_s \circ R_{\omega_s},
\end{equation}
where $\omega_s\in \Omega^2(\D(M,N))$ is the family of 2-forms on the deformation space given by
\[\omega_s= - \d \int_0^s \varphi_{u}^*\D(\alpha)\ \d u
= - \d \int_0^s \f{\varphi_{u}^*\kappa^*\alpha}{t-u}\ \d u\]
(recall that $\varphi_s^* t=t-s$).
Using $\kappa\circ \varphi_{u}\circ j=\psi\circ m_{-u}$ (see Lemma~\ref{lem:later}) and $j^*t=0$, this gives
\begin{equation}\label{eq:jomegas}
 j^*\omega_s= - \d \int_0^s \f{m_{-u}^*\psi^*\alpha}{-u}\ \d u
=\d \int_0^{-s} \f{m_{u}^*\psi^*\alpha}{u}\ \d u.\end{equation}
On the other hand, by the definition of $\wt{\psi}$,
and similarly for $\wt{\psi}'$, Equation \eqref{eq:varpir} implies that $\wt{\psi}'=\wt{\psi}\circ R_{j^*\omega_{-1}}$. Hence, putting
$s=-1$ in \eqref{eq:jomegas} proves the proposition.
\end{proof}


\section{Local splitting of Courant algebroids}\label{sec:localsplit}
As a special case of the splitting theorem for Courant algebroids, Theorem \ref{th:split-courant}, one can consider the case that $N$ is a  `small' transversal to a leaf of $\AA$. The resulting Courant algebroid over $N$ may then be regarded as the `transverse Courant algebroid structure'. In this section, we explain how to construct a `linear approximation' to
this transverse structure. Just as in the case of  Poisson structures \cite{wei:loc} or Lie algebroids, one may then pose the problem of  linearizability of this transverse Courant algebroid structure.

\subsection{Local splitting of Courant algebroids}\label{subsec:localsplitting}
The splitting result for Courant algebroids, Theorem \ref{th:split-courant}, implies the following local version:
\begin{corollary}\label{cor:split}
 A Courant algebroid $\AA$ over $M$ is isomorphic, near any given $m\in M$,
 to a product of the standard Courant algebroid over the
 leaf through $m$, and a Courant algebroid over a transverse submanifold
 such that its anchor vanishes at $m$.
\end{corollary}
\begin{proof}
Let $\P=\a(\AA_m)\subset T_mM$, the tangent space to the leaf through $m$, and pick a
submanifold $N$ such that $T_mN$ is a vector space complement
to $\P$ in $T_mM$.  Taking $N$ to be sufficiently small, the normal bundle is trivial, and, following \eqref{eq:simpl}, we get the local model $i^!\AA\times \T\P$, where $i^!\AA$ is a Courant algebroid over $N$ with vanishing anchor at $m$.
\end{proof}
Given any leaf $\mathcal{O}$ of the Courant algebroid there is a well-defined germ of a \emph{transverse Courant algebroid structure}; that is, the induced Courant algebroid structures on transversals through  any two points of $\mathcal{\O}$ are locally isomorphic.
See \cite[Section 3.8]{bur:spl} for details in the case of Lie algebroids; the argument carries over to Courant algebroids with straightforward modifications.

By the splitting result in Corollary \ref{cor:split}, the local study of Courant algebroids can be reduced to the study of Courant algebroids around critical points. Note that if $\AA$
has a critical point at $m\in M$, then $\dd=\AA_m$  is a Courant algebroid over a point, i.e. a metrized  Lie algebra. For $\sigma\in \Gamma(\AA)$, the vector field $\a(\sigma)$
vanishes at $m\in M$, and has a linear approximation $\nu(\a(\sigma))\in\mf{X}(T_mM)$ depending only on $\zeta=\sigma|_m$. Viewing
the linear vector field on $T_mM$ as an endomorphism of $T_mM$, this defines a
$\dd$-representation on $T_mM$. It turns out that this $\dd$-action has co-isotropic stabilizers, and the resulting action Courant algebroid
\[ T_mM\times \dd\]
is the \emph{linear approximation of $\AA$ near $m$}, in a sense to be made precise in the
following section. We say that a Courant algebroid $\AA$ with $\a_m=0$ is \emph{linearizable} at $m$ if it is isomorphic to the Courant algebroid $T_mM\times \dd$ near $m$. Finding conditions for linearizability is the analog for Courant algebroids of the  ``linearization problem'' previously considered in the contexts of Poisson structures and Lie algebroids \cite{wei:loc,con:norsm,cra:geo,mon:lev}. Some linearizability results for Courant algebroids have been obtained by the third author in 
\cite{lim:loc}.


\subsection{Linear approximation of transverse Courant algebroid structure}
Suppose that $\AA\to M$ is a Courant algebroid, and $S\subset M$ is a submanifold
such that $\a(\AA|_S)\subset TS$. That is, $S$ intersects each leaf $\mathcal{O}$ of $\AA$ in an open subset of $\mathcal{O}$. (We use the letter $S$ to avoid confusion with the submanifold $N$ considered earlier, which is transverse to the leaves.) Then $\AA|_S$ is again a Courant algebroid.
The normal bundle $\nu(\AA,\AA|_S)\to \nu(M,S)$ is simply the pullback (as a vector bundle) of $\AA|_S$ under the base projection $\nu(M,S)\to S$. It hence inherits a metric
by pull-back from $\AA|_S$. The map $\a\colon \AA\to TM$ takes $\AA|_S$ to $TS$, hence it defines an anchor map
\[ \nu(\a)\colon \nu(\AA,\AA|_S)\to \nu(TM,TS)\cong T\nu(M,S).\]
\begin{proposition} The vector bundle
\[ \nu(\AA,\AA|_S)\to \nu(M,S)\]
has a  Courant algebroid structure, with the metric and anchor as described above, and
with the unique Courant bracket such that the pull-back map $\Gamma(\AA|_S)\to
\Gamma(\nu(\AA,\AA|_S))$ is bracket preserving. Given a Dirac structure $E\subset \AA$,
the subbundle $\nu(E,E|_S)\subset \nu(\AA,\AA|_S)$ is a Dirac structure.
\end{proposition}
\begin{proof}
It is convenient to describe the Courant algebroid structure on $\nu(\AA,\AA|_S)$ via the deformation space $\D(\AA,\AA|_S)$. The map $\a\colon (\AA,\AA|_S)\to (TM,TS)$
defines an anchor map
\[
\D(\a)\colon \D(\AA,\AA|_S)\to \D(TM,TS)\subset T\D(M,S),
\]
where we identify $\D(TM,TS)$ with the subbundle $\ker(T\pi)$ of $T\D(M,S)$.
The metric on $\AA$, regarded as a fiberwise quadratic function $f\colon \AA\to \R,\ v\mapsto \l v,v\r$, defines a function on $\D(\AA,\AA|_S)$ which is again fiberwise quadratic.   We want to show that the obvious Courant algebroid structure on $\D(\AA,\AA|_S)|_{M\times \R^\times}=\AA\times 0_{\R^\times}$ extends to the full deformation space. Note that for all $\sigma\in \Gamma(\AA)$, the section
$\sigma\times 0\in \Gamma(\AA\times 0_{\R^\times})$ extends to a section
$\D(\sigma)$ of $\D(\AA,\AA|_S)$, defined by viewing $\sigma$ as a map of pairs $(M,S)\to (\AA,\AA|_S)$. Sections of this type generate the   entire space of sections as a module over $C^\infty(\D(M,S))$. Hence, to show that the Courant bracket of $\AA\times 0_{\R^\times}$ extends to all of $\D(\AA,\AA|_S)$, it suffices to show that the bracket of
any two sections of this type extends. But this is immediate since $\Cour{\D(\sigma_1),\D(\sigma_2)}:=\D(\Cour{\sigma_1,\sigma_2})$ provides the extension of $\Cour{\sigma_1\times 0,\sigma_2\times 0}
=\Cour{\sigma_1,\sigma_2}\times 0$. Since the anchor of $\D(\AA,\AA|_S)$ is tangent
to $\nu(M,S)\subset \D(M,S)$, this Courant algebroid structure restricts to $\nu(\AA,\AA|_S)$.
The restriction of the anchor and metric to $\nu(\AA,\AA|_S)$ are as described above.
Furthermore, since $\D(\sigma)|_{\nu(M,S)}$ is just the pull-back of $\sigma|_S\in \Gamma(\AA|_S)$ to $\nu(\AA,\AA|_S)$, the Courant bracket on sections of $\nu(\AA,\AA|_S)$ is such that the pullback of sections of $\AA|_S$ is bracket preserving.
\end{proof}

\begin{remark}
The map $\varrho\colon \Gamma(\AA|_S)\to \mf{X}(\nu(M,S))$, given by the composition
\[ \Gamma(\AA|_S)\to \Gamma(\nu(\AA,\AA|_S))\to \mf{X}(\nu(M,S))\]
of the pull-back followed by the anchor,
defines a \emph{Courant algebroid action} of $\AA|_S$ on
$\nu(M,S)\to S$, in the sense of \cite[Section 2.3]{lib:cou}, and it identifies $\nu(\AA,\AA|_S)$ as the corresponding \emph{action Courant algebroid}.
\end{remark}

In particular the proposition gives a `linear Courant algebroid'  over the
normal bundle to any leaf $S=\O\subset M$.  As a special case, if  $m\in M$ is a point at which the anchor of $\AA$ is zero, then the Courant algebroid $\AA|_S=\AA_m$ is a metrized  Lie algebra,
$\dd=\AA_m$, and we obtain a Courant algebroid structure on
\[ \nu(\AA,\AA_m)=T_mM\times \dd\to T_mM,\]
with the property that the pull-back map
$\dd\to \Gamma(T_mM\times \dd)$ intertwines the metrics and the Courant brackets.
By the theory of \cite{lib:cou}, this  means that $T_mM\times \dd$ is an \emph{action Courant algebroid} for a suitable $\dd$-action $\varrho$ on $T_mM$ with coisotopic stabilizers.
To describe the action, we need to specify the anchor. Note that the section $\nu(\sigma)\in \Gamma(T_mM\times \dd)$ obtained by restriction of $\D(\sigma)$ is just the `constant section' given by $\zeta=\sigma|_m$. We have
\[ \a(\zeta)=\a(\nu(\sigma))=\nu(\a(\sigma))\in\mf{X}(T_mM),\]
the linear approximation of the vector field $\a(\sigma)$. But a linear
vector field on a vector space is given by an endomorphism. Concretely, if
$Y$ is a vector field on $M$ with critical point at $m$, and $\nu(Y)$ is its linear approximation, then the  endomorphism of $T_mM$ corresponding to $\nu(Y)$ is
\[ v\mapsto [Y,Z]|_m\]
where $Z\in \mf{X}(M)$ is any vector field extending $v$.

\section{Lie bialgebroids}\label{sec:splitliebialgebroid}
\subsection{Basic definitions}
The notion of a  Lie bialgebroid was introduced by Mackenzie and Xu in \cite{mac:lie}.
A Lie bialgebroid $(E,F,\beta)$ over $M$ consists of
two Lie algebroids $E\to M,\ F\to M$, with a nondegenerate pairing,
\[ \beta\colon E\times_M F\to \R\]
(identifying $E=F^*$) such that the Lie algebroid differential $\d_F$ on $\Gamma(\wedge E)=\Gamma(\wedge F^*)$ is a derivation for the Schouten bracket of $E$. The notion turns out to be symmetric: If $(E,F,\beta)$ is a Lie bialgebroid then so is $(F,E,\beta)$. By \cite[Proposition 3.6]{mac:lie} (see also
\cite[Theorem 3.9]{lib:cou}), any Lie bialgebroid determines a Poisson bivector field $\pi\in \mf{X}^2(M)$, with associated bundle map $\pi^\sharp\colon T^*M\to TM$ given by composition
\begin{equation}\label{eq:poisson}
 T^*M\stackrel{\a_E^*}{\lra} E^*\stackrel{\cong}{\lra}F
\stackrel{\a_F}{\lra} TM,
\end{equation}
where $\a_E\colon E\to TM,\ \a_F\colon F\to TM$ are the anchor maps. In particular, $(\a_F\circ \a_E^*)^*=\a_E\circ \a_F^* = - \a_F\circ \a_E^*$.

Conversely, every Poisson manifold $(M,\pi)$ determines a Lie bialgebroid $(TM,\,T^*_\pi M,\, \beta)$,
where $T^*_\pi M$ is the cotangent algebroid  defined by the Poisson structure and $\beta$ is the canonical pairing; the bivector field associated to this Lie bialgebroid is $\pi$ itself.

\subsection{Statement of the normal form theorem}
Suppose  $N\subset  M$ is a submanifold transverse to the map $\pi^\sharp\colon T^*M\to TM$.
Then $N$ is also transverse to the anchor maps $\a_E,\a_F$ of $E,F$, because $\ran(\pi^\sharp)\subset \ran(\a_E)\cap \ran(\a_F)$ by definition of $\pi$. We obtain linear approximations $\nu(E)=p^!i^!E,\ \nu(F)=p^!i^!F$ over $\nu(M,N)$ as before, with a pairing
\[ \nu(\beta)\colon \nu(E)\times_{\nu(M,N)}\nu(F)\to \R\]
given by $\nu(\beta)(x,y)= \beta(i_!p_!(x),\ i_!p_!(y))$.
This pairing is \emph{ always degenerate} (if $\dim N<\dim M$), since the map $i_!p_!\colon \nu(E)\to E$ has a kernel. In the normal form result below, this will be corrected through a suitable 2-form
$\omega\in \Omega^2(\nu(M,N))$, by adding a term $\omega(\a_{\nu(E)}(x),\ \a_{\nu(F)}(y))$ to the right-hand side. We will denote this modified pairing by $\nu(\beta)^\omega$.

By Lemma~\ref{EL-ex}, it is always possible to choose a 1-form $\alpha\in \Omega^1(M)$ which is Euler-like for the Poisson structure $\pi$. The choice of such a 1-form determines a normal form:

\begin{theorem}[Normal form theorem for Lie bialgebroids]\label{th:split-liebialgebroids}
Let $(E,F,\beta)$ be a Lie bialgebroid, with associated Poisson structure $\pi$, and let $N\subset M$ be a submanifold transverse to the map $\pi^\sharp$. Choose a 1-form $\alpha\in \Omega^1(M)$, with $\alpha|_N=0$, such that $X=\pi^\sharp(\alpha)$ is Euler-like, determining a
 tubular neighborhood embedding $\psi\colon \nu(M,N)\to U\subset M$. Define the 2-form
\begin{equation}\label{eq:omega1}
\omega= \d \int_0^1 \f{1}{u} m_u^*\psi^*\alpha\ \d u \in \Omega^2(\nu(M,N)).\end{equation}
Then the pairing $\nu(\beta)^\omega$ is non-degenerate, and $\psi$ lifts to
an isomorphism of Lie bialgebroids
\[ (\nu(E),\nu(F),\nu(\beta)^\omega)\to (E|_U,F|_U,\beta).\]
\end{theorem}
As in the case of Courant algebroids, one also obtains a $G$-equivariant version of this result (for proper $G$-actions lifting to bialgebroid automorphisms).

The proof will be given in the following subsection. \medskip

\subsection{Proof of the normal form theorem}\label{subsec:prnormal}
By \cite{liu:ma}, the direct sum $\AA=E\oplus F$ acquires the structure of a Courant algebroid containing $E,F$ as Dirac structures; conversely, if $E,F\to M$ are transverse Dirac structures inside a Courant algebroid $\AA\to M$, then
$(E,F,\beta)$ is a Lie bialgebroid, with $\beta$ obtained by
restriction of the bilinear form  of $\AA$. One calls $(\AA,E,F)$ a \emph{Manin triple}. We shall adopt the viewpoint of defining Lie bialgebroids $(E,F,\beta)$ in terms of such Manin triples. For the Lie bialgebroid $(TM,\,T^*_\pi M,\beta)$ of a Poisson manifold, the corresponding Manin triple is
$(\T M,\,TM,\,\on{Gr}(\pi))$, where $\on{Gr}(\pi)$ is the graph of the map $\pi^\sharp\colon T^*M\to TM$.

The statement of Theorem \ref{th:split-liebialgebroids} is equivalent to the following assertion:

\begin{theorem}[Normal form theorem for Manin triples]
\label{th:split-manintriples}
 Let $(\AA,E,F)$ be a Manin triple over $M$,  with associated Poisson structure $\pi$,
and suppose $N\subset M$ is transverse to $\pi^\sharp$. Choose $\alpha\in \Omega^1(M)$
such that $\alpha|_N=0$ and $X=\pi^\sharp(\alpha)$ is Euler-like, defining a tubular neighborhood embedding $\psi\colon \nu(M,N)\to U\subset M$ and the 2-form $\omega$ (see \eqref{eq:omega1}). Then the Dirac structures $\nu(E)$ and $R_\omega(\nu(F))$ are transverse in $\nu(\AA)$, and $\psi$ lifts to an isomorphism of Manin triples
 \[ (\nu(\AA),\,\nu(E),\,R_\omega(\nu(F)))\to (\AA|_U,\,E|_U,\,F|_U).\]
\end{theorem}
\begin{proof}
Define sections
\[ \sigma= - \a_F^*\alpha\in \Gamma(F^*)=\Gamma(E),\ \ \ \
\tau= \a_E^*\alpha\in\Gamma(E^*)=\Gamma(F).\]
Then $\a(\sigma)=\a(\tau)=X$,
and
$\tau - \sigma =\a^*\alpha$.
The section $\sigma$ determines a Courant morphism $\wt{\psi}\colon \nu(\AA)\to \AA|_U$ taking $\nu(E)$ to $E|_U$, while $\tau$ determines a Courant morphism
$\wt{\psi}'\colon \nu(\AA)\to \AA|_U$ taking $\nu(F)$ to $F|_U$. By Proposition \ref{prop:gauge}, these are related by a gauge transformation by the 2-form $\omega$:
\[ \wt{\psi}'=\wt{\psi}\circ R_\omega.\]
It follows that
\[ F|_U=\wt{\psi}'(\nu(F))=\wt{\psi}(R_\omega(\nu(F))).\qedhere\]
%
\end{proof}


%
\begin{remark}
The normal form theorem for Manin triples, as stated in Theorem~\ref{th:split-manintriples},   admits a
generalization to Dirac structures $E,F\subset \AA$ that are not necessarily transverse, but satisfy the weaker condition $E\cap\, F\cap\, \ker(\a)=0$. In terms of
the  canonical Courant morphism \cite{lib:cou}
$S\colon \T M \da \AA\times \bar{\AA}$, where $v+\mu\sim_S (x,y)$ if and only if $v=\a(x)$ and $\a^*\mu=x-y$, this condition means that for all
$x\in E,\ y\in F$, one has  that
$v+\mu\sim_S (x,y)$ if and only if $v+\mu=0$.
As a consequence, the backward image $L=S^!(E\times F)\subset \T M$
is a \emph{Dirac structure}, which is a Poisson structure if and only if $E$ is transverse to $F$.
By construction,
any section $X+\alpha\in\Gamma(L)$
is $S$-related to a unique section
 of $E\times F$ along the diagonal.
 But this is just the
 same as a section $(\sigma,\tau)\in \Gamma(E\oplus F)$, where $(\sigma(m),\tau(m))\in E_m\oplus F_m$
is the unique element that is $S$-related to $(X_m+\alpha_m)$. By definition
of $S$, these sections satisfy
\[ \a(\sigma)=\a(\tau)=X,\ \ \
\sigma-\tau=\a^*\alpha.
\]
Given a submanifold $N$ transverse to the anchor of $L$, we may choose an Euler-like section
$X+\alpha\in \Gamma(L)$ with respect to $N$. Proceeding  as in the
transverse case, we obtain a normal form for $(E,F)$.
\end{remark}


\subsection{Local splitting theorem for Lie bialgebroids}
The assumption that $N$ is transverse to $\pi^\sharp$ holds in particular if $N$ is  \emph{cosymplectic} \cite{wei:loc}. Recall that a submanifold $N$ of a Poisson manifold $(M,\pi)$ is cosymplectic  if and only if the composition
\begin{equation}\label{eq:cosympl} \nu(M,N)^*\to T^*M|_N\stackrel{\pi^\sharp}{\lra} TM|_N\to \nu(M,N)
\end{equation}
is an isomorphism. This implies that $TN\subset TM|_N$ has a canonical complement $\pi^\sharp(\on{ann}(TN))$, which we may identify with $\nu(M,N)$,  giving a direct sum decomposition
\begin{equation}\label{eq:dirsum}
TM|_N=TN\oplus \nu(M,N).
\end{equation}
\begin{lemma} Let $(E,F,\beta)$ be a Lie bialgebroid over $M$, and suppose $N\subset M$ is cosymplectic for the associated Poisson structure $\pi$. Then $\beta$ restricts to a non-degenerate pairing $i^!\beta$ between $i^!E\subset E,\ i^!F\subset F$, defining a pull-back Lie bialgebroid $(i^!E,i^!F,i^!\beta)$ over $N$.  \end{lemma}
\begin{proof} The image $\a_F^*(\on{ann}(TN))\subset F^*|_N\cong E|_N$ is a complement $E'$ to $\a_E^{-1}(TN)\subset E|_N$, since $\a_E(E')=\pi^\sharp(\on{ann}(TN))=\nu(M,N)$. A similar argument applies to $F$. This defines splittings
\[
E|_N=\a_E^{-1}(TN)\oplus E',\ \ \ \  F|_N=\a_F^{-1}(TN)\oplus F'.
\]
Since the metric defines a non-degenerate pairing between $E$ and $F$, and since $E'$ is orthogonal to $\a_F^{-1}(TN)$, it follows that the metric defines a non-degenerate pairing between $\a_E^{-1}(TN)$ and $\a_F^{-1}(TN)$. Therefore, the pairing between $i^!E$ and $i^!F$ is non-degenerate as well.  \end{proof}
Given a Lie bialgebroid $(E,F,\beta)$ with associated Poisson structure $\pi$, and any point $m\in M$, let $S\subset M$ be the symplectic leaf through $m$, and $i\colon N\subset M$ a transverse submanifold of complementary dimension. Taking $N$ smaller if necessary, it is a cosymplectic submanifold, hence the \emph{transverse Lie bialgebroid} $(i^!E,i^!F,i^!\beta)$ over $N$ is defined.  Letting $\omega_S$ denote the symplectic form on $S$, we can also consider the Lie bialgebroid  $(TS,TS,\omega_S)$. Theorem \ref{th:split-liebialgebroids} (or the equivalent  Theorem \ref{th:split-manintriples})
has the following consequence:
\begin{corollary}\label{cor:liebisplit}  A Lie bialgebroid $(E,F,\beta)$ over $M$ is isomorphic, near any given point $m\in M$,  to a product of the standard Lie bialgebroid $(TS,TS,\omega_S)$ of the symplectic leaf through  $S$ and the transverse  Lie bialgebroid $(i^!E,i^!F,i^!\beta)$ over $N$. \end{corollary}
\begin{proof}
We will work with the interpretation as a Manin triple
$(\AA,E,F)$ with $\AA=E\oplus F$, as in
Theorem \ref{th:split-manintriples}.  By Theorem \ref{th:split-manintriples}, around $m$ there is an isomorphism of Manin triples $(\AA,E,F)\to (\nu(\AA),\nu(E),\nu(F)^\omega)$, where $\omega\in \Omega^2(\nu(M,N))$ is the closed 2-form \eqref{eq:omega1}. One can arrange (by suitably choosing $\alpha$ in \eqref{eq:omega1})  that $\omega|_N$ has kernel $TN$, and its restriction to $\nu(M,N)\subset T\nu(M,N)|_N$ coincides with the symplectic form on $\nu(M,N)\subset TM|_N$ induced by $\pi|_N$;
see  \cite[Lemma 6.2]{bur:spl}.

Let $\mathsf{P}=T_mS\subset T_mM$. Taking $N$ sufficiently small,  we may choose an isomorphism $\nu(M,N)\cong N\times \mathsf{P}$, extending the given identification
$\nu(M,N)_m\cong \mathsf{P}$ at $m$. On $N\times \mathsf{P}$, we have the closed 2-form $\omega_S$ (viewed as a 2-form on the product via pullback
from the second factor), but also the 2-form  $\omega$ (by pullback from $\nu(M,N)$). These coincide at the point $m=(m,0)\in N\times \mathsf{P}$. Define a family of closed 2-forms by linear interpolation, $\omega_s=(1-s)\omega+s\omega_S$. Then the restriction of
$\omega_s$ to the tangent space at $m$ is constant; in particular, there is a neighborhood of
$m$ on which $\nu(F)^{\omega_s}$ remains transverse to $\nu(E)$. The Moser method for Manin triples, proven in Appendix~\ref{app:Moser},
defines a Courant automorphism of $\nu(\AA)=i^!\AA\times \T S$ near $m$, preserving $\nu(E)=i^!E\times TS$, and taking
$\nu(F)^{\omega}$ to $\nu(F)^{\omega_S}=i^!F\times TS^{\omega_S}$.
\end{proof}

In the special case that the Poisson bivector field $\pi$ is itself non-degenerate, defining a symplectic structure $\omega$, we may take $N=\pt$, $S=M$. In this case, $i^!E=\g,\ i^!F=\h$ is a Manin pair
$(\g,\h,\beta_0)$ of Lie algebras, with a non-degenerate pairing $\beta_0$ obtained by restriction of $\beta$. We may regard $(TM,TM,\omega)$ as a Lie bialgebroid, and the corollary says that $(E,F,\beta)$ is isomorphic, near $m$, to the product of Lie bialgebroids
\[ (\g,\h,\beta_0)\times (TM,TM,\omega).\]

Note also that the Weinstein splitting theorem \cite{wei:loc} for Poisson manifolds $(M,\pi)$ is a special case
of Corollary \ref{cor:liebisplit}, by taking $(E,F,\beta)$ to be the Lie bialgebroid $(TM,\,T^*_\pi M,\beta)$.

\appendix

\section{Coordinates on the deformation space}\label{app:coord}
In this appendix, we will give a few more details regarding the construction of the deformation space (normal cone) $\D(M,N)=\nu(M,N)\sqcup (M\times \R^\times)$ of a manifold $M$ and submanifold $N$. References include \cite{hig:eul,hig:tan,kas:ma}.

For any function $f\in C^\infty(M)$ with $f|_N=0$, define a function $\wt{f}\colon \D(M,N)\to \R$, given by $\wt{f}(m,t)=\f{1}{t}f(m)$ on $M\times \R^\times$ and by
$TM|_N/TN\to \R,\ v\mod TN\mapsto v(f)$  on $\nu(M,N)$. The manifold structure of
$\D(M,N)$ requires that the function $\wt{f}$ is smooth. We may use these functions
to construct charts of $\D(M,N)$, as follows. Let $x_1,\ldots,x_r,y_1,\ldots,y_s\in C^\infty(U)$ be the local coordinates on a submanifold chart $U\subset M$ for $(M,N)$, that is, $U\cap N$ is given by the vanishing of the $y$-coordinates.
The deformation space $\D(U,U\cap N)\subset \D(M,N)$ is then the domain of a
chart for the deformation space, with coordinates $x_1,\ldots,x_r,\wt{y}_1,\ldots,\wt{y}_s,t$. Here $t$ is the projection $\pi\colon \D(M,N)\to \R$, regarded as a real-valued function, $\wt{y}_j$ are obtained from
$y_j$ by the procedure described above, and the $x_i$ are now viewed as functions
on $U\times \R^\times$ (via projection to the first factor)
respectively on $\nu(U,U\cap N)$ (via projection to the base $U\cap N$). One verifies
(as in the references above) that the transition maps for any two such charts are smooth.

A vector field $X$ is Euler-like if and only if it has the local coordinate expression
\[ X=\sum_i a_i(x,y) \f{\p}{\p x_i}+\sum_j \big(y_j+b_j(x,y)\big)\f{\p}{\p y_j},\]
where $a_i$ vanishes for $y=0$, while $b_j$ vanishes to second order. Hence
\[ \D(X)=\sum_i a_i(x, t\wt{y}) \f{\p}{\p x_i}+\sum_j \big(\wt{y}_j+\f{1}{t} b_j(x,t\wt{y})\big)\f{\p}{\p \wt{y}_j}.\]
On the other hand, the vector field $\Theta$ is given in local $(x,y,t)$ coordinates  on $M\times \R^\times$ by $\f{\p}{\p t}$. In terms of
$(x,\wt{y},t)$ coordinates this becomes
\[ \Theta=t\f{\p}{\p t}-\sum \wt{y}_i\f{\p}{\p \wt{y_i}}.\]
Hence, for $W=\f{1}{t}(\Theta+X)$ we find
\[ W=\f{\p}{\p t}+\f{1}{t}\sum_i a_i(x, t\wt{y})\f{\p}{\p x_i}+\f{1}{t^2}\sum_j b_j(x,t\wt{y}) \f{\p}{\p \wt{y}_j}.\]
This confirms that $W$ is well-defined even for $t=0$.
For a general vector field $Y=\sum_i c_i(x,y)\f{\p}{\p x_i}+\sum_j d_j(x,y)\f{\p}{\p y_j}$,
the vector field $Y\times 0$ on $M\times \R^\times$ has the coordinate expression
\[ Y\times 0=\sum_i c_i(x, t\wt{y}) \f{\p}{\p x_i}+\f{1}{t} \sum_j d_j(x,t\wt{y})\f{\p}{\p \wt{y}_j}.\]
We see again that $t(Y\times 0)$ extends to a vector field $\wh{Y}$ on the deformation space, while $Y\times 0$ extends to a vector field $\D(Y)$
if and only if  $Y$ is tangent to $N$.

In a similar fashion, for a $k$-form $\alpha=
\sum_{|I|+|J|=k} f_{IJ}(x,y)\d x^I \d y^J$ (using multi-index notation), the form $t^{-1}(\alpha\times 0)$ on $M\times \R^\times$ extends to a $k$-form $\D(\alpha)$ on the deformation space (i.e., $\kappa^*\alpha$ is divisible by $t$) if $\alpha|_N=0$. 
In this case, the pullback
of $\D(\alpha)$ to $\nu(M,N)$ is the linear approximation $\nu(\alpha)$.

\section{The Moser method for Manin triples}\label{app:Moser}
Let $\AA\to M$ be a Courant algebroid. Let $\sigma_s,\tau_s\in \Gamma(\AA)$ be 1-parameter families of sections, with $\tau_s=\sigma_s+\a^*\beta_s$ for a family of 1-forms $\beta_s\in \Omega^1(M)$,
and let $\wt{\psi}_s,\wt{\phi}_s$ be the (local) flows on $\AA$ generated by these sections. Their base flows coincide, and are given by the flow of the time dependent vector field $X_s=\a(\sigma_s)=\a(\tau_s)$.
Proposition \ref{prop:gaugerelated} also applies to time-dependent sections, with the same proof,
and shows that
\begin{equation}\label{eq:betaeqn}
\wt{\psi}_s=R_{\varpi_s}\circ \wt{\phi}_s,\ \ \ \
 \varpi_s=-\d \int_0^s (\phi_u)_*\beta_u \ \d u.\end{equation}
We will use this to prove:
\begin{theorem}[Moser method for Manin triples]
Let $\AA\to M$ be a Courant algebroid with a Dirac structure $E\subset \AA$. Let
$F_s,\ s\in\R$ be a family of Dirac structures transverse to $E$, depending smoothly on $s$,
and which are gauge equivalent in the sense that
\[ F_s={R}_{\nu_s}(F_0)\]
for a smooth family of closed 2-forms $\nu_s\in \Om^2(M)$, with $\nu_0=0$.
Denote by $\pi_s\in\mf{X}^2(M)$ the resulting family of Poisson structures. Suppose furthermore that
$\f{d}{d s}\nu_s=-\d\alpha_s$ for a smooth family of 1-forms
$\alpha_s$, and such that the vector field $\pi_s^\sharp(\alpha_s)$ is complete. Then there is a family of Courant automorphisms $\wt{\phi}_s$ such that
\[ \wt{\phi}_s(E)=E,\ \ \ \wt{\phi}_s(F_s)=F_0.\]
\end{theorem}
In practical applications (as in the proof of Cor.~\ref{cor:liebisplit}), the completeness assumptions are ensured by multiplying with suitable bump functions.
\begin{proof}
(See \cite{me:poilec} for a similar argument in the context of Poisson geometry.)
Let $\sigma_s \in \Gamma(E)$ be the component of $-\a^*(\alpha_s)\in \Gamma(\AA)$ under the decomposition $\AA=E\oplus F_s$. By assumption, the time dependent vector field $X_s=\a(\sigma_s)=\pi^\sharp_s(\alpha_s)$ is complete, defining a flow $\phi_s$ on $M$, and
hence the time dependent section $\sigma_s$ generates a flow $\wt{\phi}_s$.
 Since $\sigma_s$ are sections of $E$, we have that $\wt{\phi}_s(E)=E$. To prove that this flow takes $F_s$ to $F_0$, note first that
\[ \sigma_s+\a^*(\alpha_s)\]
is a family of sections of $F_s={R}_{\nu_s}(F_0)$. Applying ${R}_{-\nu_s}$, we obtain a section of
$F_0$. But
\begin{align*}
 {R}_{-\nu_s}(\sigma_s+\a^*(\alpha_s))&= \sigma_s+\a^*(\alpha_s) -\a^*\circ  \nu_s^\flat \circ \a
(\sigma_s+\a^*(\alpha_s))\\
&= \sigma_s+\a^*(\alpha_s-\nu_s^\flat (X_s))\\
&=\sigma_s+\a^*\beta_s
\end{align*}
with the family of 1-forms $\beta_s=\alpha_s-\iota_{X_s} \nu_s\in \Omega^1(M)$.
The flow of the family of sections $\sigma_s+\a^*\beta_s$ is $\wt{\psi}_s=R_{\varpi_s}\circ \wt{\phi}_s$, with
$\varpi_s$ obtained by integrating
\[ (\phi_s)_*\d \beta_s=-\f{d}{d s}\varpi_s,\]
see \eqref{eq:betaeqn}. But
\[ (\phi_s)_*\d\beta_s=(\phi_s)_*(\d\alpha_s-\L_{X_s}\nu_s)=-\f{d}{d s}((\phi_s)_*\nu_s)\]
shows that $\varpi_s=(\phi_s)_*\nu_s$. Since $\sigma_s+\a^*\beta_s$ are sections of $F_0$, the flow
$\wt{\psi}_s$ preserves $F_0$. We conclude that
\[
F_0=\wt{\psi}_s(F_0)={R}_{(\phi_s)_*\nu_s}\circ \wt{\phi}_s (F_0)=
\wt{\phi_s}\circ {R}_{\nu_s}(F_0)=\wt{\phi_s}(F_s).\qedhere \]
\end{proof}

We remark that the Moser method for Poisson manifolds (see, for example, \cite{me:poilec}) is a special case, where $\AA=\T M,\ E=TM$ and $F_s=\on{Gr}(\pi_s)$ for a family of gauge equivalent Poisson structures.

\bibliographystyle{amsplain}

\def\cprime{$'$} \def\polhk#1{\setbox0=\hbox{#1}{\ooalign{\hidewidth
  \lower1.5ex\hbox{`}\hidewidth\crcr\unhbox0}}} \def\cprime{$'$}
  \def\cprime{$'$} \def\cprime{$'$} \def\cprime{$'$} \def\cprime{$'$}
  \def\polhk#1{\setbox0=\hbox{#1}{\ooalign{\hidewidth
  \lower1.5ex\hbox{`}\hidewidth\crcr\unhbox0}}} \def\cprime{$'$}
  \def\cprime{$'$} \def\cprime{$'$} \def\cprime{$'$} \def\cprime{$'$}
\providecommand{\bysame}{\leavevmode\hbox to3em{\hrulefill}\thinspace}
\providecommand{\MR}{\relax\ifhmode\unskip\space\fi MR }
\providecommand{\MRhref}[2]{%
  \href{http://www.ams.org/mathscinet-getitem?mr=#1}{#2}
}
\providecommand{\href}[2]{#2}

\end{document}